\documentclass{article}
\usepackage[margin=1in]{geometry}
\usepackage{enumitem}
\usepackage{hyperref}
\usepackage{amsmath,amsfonts,amsthm,amssymb,bbm}
\usepackage{graphicx,color,dsfont}

\newtheorem{theorem}{Theorem}

\newtheorem{corollary}{Corollary}

\definecolor{ao}{rgb}{0.0, 0.5, 0.0}
\definecolor{darkpastelgreen}{rgb}{0.01, 0.75, 0.24}
\definecolor{auburn}{rgb}{0.43, 0.21, 0.1}
\definecolor{armygreen}{rgb}{0.29, 0.33, 0.13}

 \newtheorem{thm}{Theorem}[section]
 \newtheorem{cor}[thm]{Corollary}
 \newtheorem{lem}[thm]{Lemma}
 
 \theoremstyle{definition}
 \newtheorem{defn}[thm]{Definition}
 \theoremstyle{remark}
 \newtheorem{rem}[thm]{Remark}
 
 \numberwithin{equation}{section}

\newcommand{\rone}{{\mathds R}}

\newcommand{\cone}{\mathds C}
\newcommand{\zone}{{\mathds Z}}
\def\im{{\rm i}}
\newcommand{\R}{{\mathbb R}}

\begin{document}

\title{On homogeneous Besov  spaces for $1D$ Hamiltonians without zero resonance }

\author{Vladimir Georgiev, \thanks{The first author was supported in part by Contract FIRB " Dinamiche Dispersive: Analisi di Fourier e Metodi Variazionali.", 2012, by INDAM, GNAMPA - Gruppo Nazionale per l'Analisi Matematica, la Probabilit\`{a} e le loro Applicazioni and by Institute of Mathematics and Informatics, Bulgarian Academy of Sciences.}
\\ Department of Mathematics, University of Pisa, Largo B. Pontecorvo 5, \\  Pisa,
56127 Italy, \\
 georgiev@dm.unipi.it \\
\and \\
 Anna Rita Giammetta, \thanks{ The second  author was supported in part by Contract FIRB " Dinamiche Dispersive: Analisi di Fourier e Metodi Variazionali.", 2012 and by INDAM, GNAMPA.}\\ Department of Mathematics, University of Pisa, Largo B. Pontecorvo 5, \\ Pisa,
56127 Italy, \\
giammetta@mail.dm.unipi.it}



\maketitle

\begin{abstract}
We consider 1-D Laplace operator with short range potential $V(x),$ such that $$(1+|x|)^\gamma V(x) \in L^1(\rone), \ \  \gamma > 1.$$  We study the equivalence of classical
homogeneous  Besov type spaces  $\dot{B}^s_p(\rone)$, $p \in (1,\infty)$ and the corresponding perturbed homogeneous Besov spaces associated with the perturbed Hamiltonian $\mathcal{H}= -\partial_x^2 + V(x)$ on the real line. It is shown that the assumptions $1/p < \gamma -1$ and zero is not a resonance   guarantee that the perturbed and unperturbed homogeneous Besov norms of order $s \in [0,1/p)$ are equivalent. As a corollary,  the corresponding wave operators leave classical homogeneous Besov spaces of order $s \in [0,1/p)$ invariant.
\end{abstract}

\section{ Introduction.}

The wave operator methods  have been  used  frequently in the study of the  evolution flow generated by  Hamiltonians, typically considered as perturbations of free Hamiltonians.
The wave operators are defined by the relation
  $$ W_\pm = s - \lim_{t \to \pm \infty} e^{it\mathcal{H}} e^{-it\mathcal{H}_0},$$
  where $\mathcal{H}_0$ is the free Hamiltonian (self-adjoint non-negative operator), $\mathcal{H}$ is the perturbed one and $s - \lim$ means strong limit.
  The existence and completeness of the wave operators in  standard Hilbert space (typically Lebesgue space $L^2$) in case of short range perturbations is well known (see \cite{LPh64}, \cite{RSI78}, \cite{HII} and the references therein).

The  functional calculus for the perturbed non-negative operator $\mathcal{H}$ can be introduced with a relation involving $W_\pm$
\begin{equation}\label{eq.III2}
   g(\mathcal{H}) = W_+ g(\mathcal{H}_0)W_+^*= W_- g(\mathcal{H}_0)W_-^*,
\end{equation}
for any function $g \in L^\infty_{loc}(0,\infty).$ Moreover, the wave operators map unperturbed Sobolev spaces in the perturbed ones,
$$ W_\pm :  D(\mathcal{H}_0^{s/2}) \to  D(\mathcal{H}^{s/2}) $$
and we have the equivalence of the Sobolev norms (see \cite{Y95} for more general Sobolev norms)
\begin{equation}\label{eq.II1}
    \|\mathcal{H}^{s/2}f\|_{L^2(\rone)} + \|f\|_{L^2(\rone)} \sim \|\mathcal{H}_0^{s/2}f\|_{L^2(\rone)} + \|f\|_{L^2(\rone)}.
\end{equation}
The study of the dispersive properties of the evolution flow in some cases of short range perturbations shows (see \cite{CGV}) that we have stronger equivalence between homogeneous Sobolev norms
\begin{equation}\label{eq.II2}
    \| \mathcal{H}^{s/2}f \|_{L^2(\rone^n)}  \sim \|\mathcal{H}_0^{s/2}f\|_{L^2(\rone^n)},
\end{equation}
provided $s < n/2.$

Our first goal in this work is to show that the requirement $s<n/2$ is optimal at least for $n=1,2$ and $\mathcal{H}_0= -\Delta=-\partial_{x}^2$, i.e. we shall prove the following result:

\begin{theorem} \label{l.co1} If $n=1,2,$ and $V (x)$ is a positive potential such that
\begin{equation}\label{eq.C1a1}
   \int_{\rone^n} V^{n/2}(x) dx \leq C < \infty,
\end{equation}
 then \eqref{eq.II2} with $s=n/2$ is not true.
\end{theorem}

  The mapping properties for the case of Sobolev spaces $W^s_p(\rone^n)$ are studied in (\cite{Y95}, \cite{W0})
     and they show examples of spaces invariant under the action of the wave operators.

Our unperturbed Hamiltonian $\mathcal{H}_0$ is the self-adjoint realization of $-\partial_x^2$ on the real line $\rone.$ The perturbed Hamiltonian
is $\mathcal{H}=-\partial_x^2+V(x).$ The results in \cite{W0} deal with short range assumptions that guarantee $W^{k}_p(\rone)$ boundedness of $W_\pm.$
The $L^p(\rone)$ boundedness is studied in \cite{DF06}.

Our key goal in this work is to study how classical homogeneous Besov spaces $\dot{B}^s_p(\rone)$ are transformed under the action  the wave operators.

The splitting property \eqref{eq.III2}
implies
that
$$ W_\pm :  \dot{B}^s_p(\rone) \Longrightarrow \dot{B}^s_{p,\mathcal{H}}(\rone), \ \forall s \geq 0, \ 1 < p < \infty,$$
where $\dot{B}^s_{p,\mathcal{H}}(\rone)$ is the perturbed Besov space generated by the Hamiltonian $\mathcal{H}.$
More precisely, $\dot{B}^s_{p, \mathcal{H}}(\rone)$ is the  homogeneous Besov spaces associated with the perturbed Hamiltonian $ \mathcal{H}= -\partial_x^2+V$ as the closure of $S(\rone)$ functions $f$ with respect to the norm
\begin{equation}\label{eq:BS2} \begin{aligned}
   \|  f\|_{ \dot{B}^s_{p,\mathcal{H}}(\rone)} =
 \left( \sum_{j=-\infty}^\infty 2^{2js} \left\| \varphi \left(
\frac{ \sqrt{\mathcal{H}}}{2^j} \right)f\right\|^2_{L^p(\rone)} \right)^{1/2} .
\end{aligned}\end{equation}

Here and below
 $\varphi(\tau) \in C_0^\infty(\rone \smallsetminus 0)$ is a non-negative  even function, such that
 $$ \sum_{j \in \zone} \varphi \left(
\frac{ s}{2^j} \right) = 1 \ , \ \ \forall s \in \rone \setminus  0 .$$

Some basic properties of these Besov spaces and the independence of the Besov space of the choice of the Paley-Littlewood function $\varphi$ can be found in \cite{Ze10}.

The equivalence of the homogeneous Besov norm
\begin{equation}\label{eq.IMi}
  \sum_{j=-\infty}^\infty 2^{2js} \left\| \varphi \left(
\frac{ \sqrt{\mathcal{H}}}{2^j} \right)f\right\|^2_{L^p(\rone)} \sim  \sum_{j=-\infty}^\infty 2^{2js} \left\| \varphi \left(
\frac{ \sqrt{\mathcal{H}_0}}{2^j} \right)f\right\|^2_{L^p(\rone)} ,
\end{equation}
imply that the homogeneous Besov space $\dot{B}^s_p(\R) $ is also invariant under the action of the wave operators $W_\pm.$
The natural restriction $0 \leq s < 1/p$ can be justified  by Theorem \ref{l.co1}.

Our approach to establish \eqref{eq.IMi} is based on establishing estimates of this kind
\begin{equation}\label{eq.III5}
    \left\| \varphi \left(
\frac{ \sqrt{\mathcal{H}}}{2^j} \right) \varphi \left(
\frac{ \sqrt{\mathcal{H}_0}}{2^k} \right)f\right\|_{L^p(\rone)} \leq  \frac{C}{2^{|j-k|s}} \|f\|_{L^p(\rone)}, \ \ \forall \ s >0, \ s < \frac{1}{p} ,
\end{equation}
where $j,k\in\mathbb{Z}$ will satisfy certain relations.

\section{Assumptions and main results}


We shall assume that the potential  $V:\rone\to \rone$ is a real-valued potential, $V \in L^1(\R)$ and $V$ is decaying sufficiently rapidly at infinity, namely following \cite{W} we require
  \begin{equation}\label{V6}
    \|\langle x \rangle^\gamma V\|_{L^1(\rone)}  < \infty, \   \gamma > 1+1/p, \ 1<p<\infty,
\end{equation}
or equivalently we assume $ V\in L^1_{\gamma}(\rone) $, where
$$ \  L^1_\gamma(\rone) = \{f \in L^1_{loc}( \rone) ; \langle x \rangle^\gamma f(x) \in L^1(\rone) \}, \ \langle x \rangle^2 = 1+x^2. $$

We shall impose for simplicity in this work the assumption that  the point spectrum of  $\mathcal{H}=-\partial_x^2+V(x)$ is empty, i.e.
\begin{equation}\label{V6a}
   \mathcal{H} f - z f = 0, \ f \in L^2(\rone), \ z \in \cone \Longrightarrow f=0.
\end{equation}

Moreover, we are looking for appropriate decomposition  of the kernel of the  Paley-Littlewood localization operator
\begin{equation}\label{eq.If0}
    \varphi \left(
\frac{ \sqrt{\mathcal{H}}}{2^j}  \right),
\end{equation}
 where  $\varphi(\tau) \in C_0^\infty(\rone \setminus 0)$ is an even function and $j \in \zone.$
We plan to decompose the kernel of the operator \eqref{eq.If0} into a leading term, involving similar localization operators for the unperturbed Hamiltonian $\mathcal{H}_0$
\begin{equation}\label{eq.If1}
\left| \varphi \left(
\frac{ \sqrt{\mathcal{H}_0}}{2^j}  \right)(x,y) \right| \leq \frac{C 2^j}{\langle 2^j (x-y) \rangle^2} \ \ \forall j \in \zone,
\end{equation}
and a remainder satisfying better kernel estimates.

The existence of the wave operators $W_\pm$ is well known according to the results in
\cite{W0}, \cite{AY}, \cite{DF06}, so $W_\pm$ are well  defined operators in $L^p(\rone),$ $1<p<\infty.$

The splitting property
$$ \mathcal{H }W_\pm = W_\pm \mathcal{H}_0$$
implies
that
$$ W_\pm :  \dot{B}^s_p(\rone) \to \dot{B}^s_{p,\mathcal{H}}(\rone), \ \forall s \geq 0, \ 1 < p < \infty.$$

The functional calculus for the perturbed operator $\mathcal{H}$
can be defined as follows
$$ g(\mathcal{H}) = W_+ g(\mathcal{H}_0)W_+^*= W_- g(\mathcal{H}_0)W_-^*$$
for any function $g \in L^\infty_{loc}(\rone).$

The functional calculus enables one to introduce a  Paley-Littlewood partition of unity
$$ 1 = \sum_{j \in {\mathbf Z}}  \varphi \left(\frac{t}{ 2^{ j}}\right) , \ t > 0$$
for an appropriate non-negative cutoff  $\varphi \in C^\infty _0({\rone} _+),$ such that ${\rm supp} \varphi \subseteq [1/2,2].$

The homogeneous Besov spaces $\dot{B}^s_{p}(\rone)$  for $p$, $1 \leq p \leq  \infty$ and $s \geq 0$ can be defined as the closure of $S(\rone)$ functions $f$ with respect to the norm
\begin{equation}\label{eq:BS1} \begin{aligned}
   \|  f\|_{ \dot{B}^s_{p}(\rone)} =
 \left( \sum_{j=-\infty}^\infty 2^{2js} \left\| \varphi \left(
\frac{ \sqrt{\mathcal{H}_0}}{2^j} \right)f\right\|^2_{L^p(\rone)} \right)^{1/2} .
\end{aligned}\end{equation}

The perturbed Besov spaces  $\dot{B}^s_{p, \mathcal{H}}(\rone)$ have been already defined in \eqref{eq:BS2}.

 Using the classical result due to Weder \cite{W0} one can derive the following $L^p$ estimate.
 \begin{equation}\label{eq.Ber1}
  \left\| \varphi \left(
\frac{ \sqrt{\mathcal{H}}}{M} \right) f\right\|_{L^p(\rone)} \leq C  \|f\|_{L^p(\rone)}
 \end{equation}
 for $M>0$, $f \in S(\rone)$, $1<p<\infty$.

First we prove the following high energy kernel estimate needed in the proof of the e\-qui\-va\-len\-ce of homogeneous Besov norms.

\begin{lem}\label{Ber1m}
Suppose the condition \eqref{V6} is fulfilled with $\gamma >1+1/p$ and  the operator $\mathcal{H}$ has no point spectrum.
If  $\varphi$ is an even non-negative function, such that  $\varphi \in C^\infty _0({\mathbf R} \setminus \{0\}),$ then for any $M \in [1,\infty)$, $\sigma \in (0,1)\cap (0,\gamma-1]$ we have
\begin{align}\label{eq.Ber2m}
   \left| \varphi \left( \frac{\sqrt{\mathcal{H}}}{M} \right)(x,y) -\varphi \left(\frac{\sqrt{\mathcal{H}_0}}{M} \right)(x,y)\right| \leq  \\ \nonumber \leq C  \left(\sum_{\pm} \frac{1}{\langle M(x\pm y) \rangle^{ \sigma}}\right)\left(\frac{1}{\langle x \rangle^{\gamma-\sigma}} + \frac{1}{\langle y \rangle^{\gamma-\sigma}}\right).
\end{align}
\end{lem}

The proof is based on careful evaluation of the kernel of the operator $ \varphi \left(
 \sqrt{\mathcal{H}}/M\right),$ having the representation
\begin{alignat}{2}\label{eq.IFC1}
  \varphi \left( \sqrt{\mathcal{H}}/M\right)(x,y) &=  - \frac  1 {2 \pi} \int_\rone \varphi \left(\tau/M\right) T(\tau) f_+(y,\tau) f_-(x,\tau)  d\tau,\ \ \mbox{ when} \ \ x < y,
\end{alignat}
and
\begin{alignat}{2}\label{eq.IFC1i}
  \varphi \left( \sqrt{\mathcal{H}}/M\right)(x,y) &= - \frac  1 {2 \pi} \int_\rone \varphi \left(\tau/M\right) T(\tau) f_-(y,\tau) f_+(x,\tau)  d\tau, \ \ \mbox{otherwise.}
\end{alignat}
Here and below $T(\tau)$ is the transmission coefficient (see \eqref{eq:kernel35}  for its definition). Moreover, $$ f_\pm (x,\tau) = e^{\pm \im  \tau x} m_\pm (x,\tau), $$
and $f_\pm(x,\tau)$ are the Jost functions (see Section 2 in \cite{DeiTru}) satisfying the integral equations (Marchenko type equations)
\begin{eqnarray}\label{eq.Igra1n}
       m_+(x,\tau) = 1+ \int_x^\infty D(t-x,\tau) V(t) m_+(t,\tau) dt,
    \\ \nonumber
       m_-(x,\tau) = 1+ \int_{-\infty}^x D(x-t,\tau) V(t) m_-(t,\tau) dt,
    \end{eqnarray}
with
\begin{equation}\label{eq.IAE3}
     D(t,\tau) = \frac{e^{2it\tau}-1}{2i\tau} = \int_0^t e^{2iy\tau} dy.
\end{equation}
The kernel $\varphi \left( \sqrt{\mathcal{H}}/M\right)(x,y)$ is symmetric, that it
\begin{alignat}{2}\label{eq.IFC1aa}
  \varphi \left( \sqrt{\mathcal{H}}/M\right)(x,y) &=  \varphi \left( \sqrt{\mathcal{H}}/M\right)(x,y).
\end{alignat}

For the low energy domain $M \in (0,1],$ we  have the following  estimate.

\begin{lem}\label{Ber1}
Suppose the condition \eqref{V6} is fulfilled with $\gamma >1+1/p$, the operator $\mathcal{H}$ has no point spectrum and $0$ is not a resonance point  for $\mathcal{H}.$
If  $\varphi$ is an even non-negative function, such that  $\varphi \in C^\infty _0({\mathbf R} \setminus \{0\}),$ then for any $M \in (0,1]$ and $\sigma \in (0,1)\cap (0,\gamma-1]$ we have
\begin{gather}\label{eq.Ber2}
   \left| \varphi \left( \frac{\sqrt{\mathcal{H}}}{M} \right)(x,y) -  K_M(x,y)\right| \leq  \\ \nonumber \leq C M \left(\sum_{\pm} \frac{1}{\langle M(x\pm y) \rangle^{ \sigma}}\right)\left(\frac{1}{\langle x \rangle^{\gamma-\sigma}} + \frac{1}{\langle y \rangle^{\gamma-\sigma}}\right),
\end{gather}
where
\begin{align}\label{eq.pasq10}
  K_M(x,y) = c\int_\R e^{-\im \tau(x - y)}\varphi\left(\frac{\tau}{M}\right) b(x,y,\tau)\,d\tau
\end{align}
with symmetric kernel $ b(x,y,\tau) = b(y,x,\tau)$ and
\begin{equation*}
b(x,y,\tau)=
\begin{cases}
T(\tau) &x<0<y,\\
(R_+(\tau)+1) e^{2i\tau x}-e^{2i\tau x}+1 & 0<x<y,\\
(R_-(\tau)+1)e^{-2i\tau y}-e^{-2i\tau y}+1 & x<y<0.
\end{cases}
\end{equation*}
\end{lem}

\begin{rem} The precise definition of the notion of resonance point at the origin is given in Definition \ref{dres} by the aid of the relation $$T(0)=0.$$
\end{rem}

Our next result treats the equivalence of the homogeneous Besov norms for the free and perturbed Hamiltonians. Here we meet the natural obstruction to cover all positive values of $s$ so we impose a condition
\begin{equation}\label{eq.HH2}
    s < \frac{1}{p},
\end{equation}
similar to the Hardy inequality restrictions.
\begin{theorem} \label{MT1}
Suppose $$V\in L^1_\gamma(\rone), \  \gamma > 1+1/p, \ 0 \leq s < 1/p , \  p \in (1,\infty),$$ the operator $\mathcal{H}$ has no point spectrum and $0$ is not a resonance for $\mathcal{H}.$ Then we have
$$  \|  f\|_{ \dot{B}^s_{p,\mathcal{H}}(\rone)} \sim \|  f\|_{ \dot{B}^s_{p}(\rone)}.$$
\end{theorem}

As immediate consequence we have the following.
\begin{corollary} \label{cMT1}
Suppose the assumptions of Theorem \ref{MT1} are fulfilled. Then for any $p \in (1,\infty),$ any $s \in [0,1/p),$ we have
$$   W_\pm :  \dot{B}^s_p(\rone) \to \dot{B}^s_{p}(\rone).$$
\end{corollary}

The authors are grateful to Atanas Stefanov for the critical remarks and discussions during the preparation of the work.

We shall present the plan of the work.

\section{Counterexample for equivalence of homogeneous Besov spaces}

In this section we consider the simplest  case $p=2$ and we shall prove Theorem \ref{l.co1}, therefore we shall show that the equivalence property
\begin{equation}\label{eq.c1}
   \|(\mathcal{H}_0+V)^{n/4} u \|_{L^2(\rone^n)} \sim  \|(\mathcal{H}_0)^{n/4}  u \|_{L^2(\rone^n)}
\end{equation}
is not true for $n=1,2.$
\begin{proof}[Proof of Theorem \ref{l.co1}]
Let us suppose that the relation \eqref{eq.c1} holds. First, we use an interpolation argument and show that
\begin{equation}\label{eq.co1}
  \| \mathcal{H}_0^a u \|_{L^2(\rone^n)}^2 \geq \| V^a u \|_{L^2(\rone^n)}^2 ,
\end{equation}
provided $0 \leq {\rm Re} a \leq 1/2.$ Indeed, we have the property
$$ \| \mathcal{H}_0^{\im b} u \|_{L^2(\rone^n)}^2 = \|  u \|_{L^2(\rone^n)}^2, \ \forall b \in \rone,$$
and
$$ \| V^{\im b} u \|_{L^2(\rone^n)}^2 = \|  u \|_{L^2(\rone)}^2, \ \forall b \in \rone,$$ so we have to check \eqref{eq.co1} only for $a=1/2.$
The equivalence of the norms \eqref{eq.co2} implies that
\begin{align*}
 \|\mathcal{H}_0^{1/2} u\|_{L^2(\rone^n)}\approx\| \left( - \Delta + V \right)^{1/2} u \|_{L^2(\rone^n)}^2 =& \langle (-\Delta + V) u, u \rangle_{L^2(\rone^n)}\\
\geq& \langle V u, u \rangle_{L^2(\rone^n)} = \| V^{1/2}u \|_{L^2(\rone^n)}^2 ,
 \end{align*}
and  we conclude that \eqref{eq.co1} is true.
Then, assuming \eqref{eq.c1} is fulfilled and applying the proved inequality with $a=n/4\leq 1/2$, we get
\begin{equation}\label{eq.co2}
   \int_{\R^n} (V(x))^{n/2} |u(x)|^2 dx \leq C \|D^{n/2}u\|_{L^2(\R^n)}^2, \  D=(-\Delta)^{1/2}.
\end{equation}
Taking $u $ in the Schwartz class $ S(\R^n) $
of rapidly decreasing function, we can apply a rescaling argument. Indeed, considering the dilation
$$ u_\lambda(x) =  u(x\lambda),$$
we find
$$   \|D^{n/2}u_\lambda\|_{L^2(\R^n)}^2 =  \underbrace{\|D^{n/2}u\|_{L^2(\R^n)}^2}_{\mbox{constant in $\lambda$}} $$
and
$$ \lim_{\lambda \searrow 0}\int_{\R^n} V^{n/2}(x) |u_\lambda(x)|^2 dx = \left(\int_{\R^n} V^{n/2}(x)  dx\right) |u(0)|^2. $$ In this way we deduce
\begin{equation}\label{eq.co2a1}
  |u(0)|^2 \left(\int_{\R^n} V^{n/2}(x)  dx\right) \leq  C \|D^{n/2}u\|_{L^2(\R^n)}^2.
\end{equation}
The homogeneous norm
$$ \|D^{n/2}u\|_{L^2(\R^n)}^2$$ is also invariant under translations, i.e. setting
$$ u^{(\tau)}(x) = u(x+\tau), $$ we have
$$ \widehat{u^{(\tau)}}(\xi) = e^{-\im \tau \xi} \widehat{u}(\xi) $$ and
$$  \|D^{n/2}u^{(\tau)}\|_{L^2(\R^n)}^2 =  \||\xi|^{n/2}\widehat{u^{(\tau)}}\|_{L^2(\R^n)}^2 = \||\xi|^{n/2}\widehat{u}\|_{L^2(\R^n)}^2 = \|D^{n/2}u\|_{L^2(\R^n)}^2,$$ so applying \eqref{eq.co2a1} with $u^{(\tau)}$ in the place of  $u$, we find
$$
  |u(\tau)|^2 \int_{\R^n} V^{n/2}(x)  dx \leq  C \|D^{n/2}u\|_{L^2(\R^n)}^2,
$$
or equivalently
\begin{equation}\label{eq.co2a2}
   \|u\|^2_{L^\infty(\R^n)} \leq C_1 \|D^{n/2}u\|_{L^2(\R^n)}^2,
\end{equation}
where
$$ C_1 = \frac{C}{\|V^{n/2}\|_{L^1(\R^n)}}.$$

The substitution  $ \phi = D^{n/2}u $ enables us to rewrite \eqref{eq.co2a2} as
\begin{equation}\label{eq.co2a3}
      \|I_{n/2}(\phi)\|^2_{L^\infty(\R^n)} \leq C_1 \|\phi\|_{L^2(\R^n)}^2,
\end{equation}
where
$$ I_\alpha(\phi)(x) = D^{-\alpha}(\phi)(x) = c \int_{\R^n} |x-y|^{-n+\alpha} \phi(y) dy, \  \alpha \in (0,n) $$
are the Riesz operators.

It is easy to show that \eqref{eq.co2a3} leads to a contradiction. Indeed, taking
$$ \phi_N(x) = \sum_{j=0}^N |x|^{-n/2}\  \mathds{1}_{2^j \leq |x| \leq  2^{j+1}} (x), $$
with $N \geq 2$ sufficiently large and being $\mathds{1}_A(x)$ the characteristic function of the set $A$, we can use the estimates
$$ I_{n/2}(\phi_N)(0) \geq  \left( \sum_{j=0}^N \int_{2^j}^{2^{j+1}} \frac{r^{n-1}dr}{r^n} \right) \geq C N$$
and
$$  \|\phi_N\|_{L^2(\R^n)}^2 = \sum_{j=0}^N \int_{2^j}^{2^{j+1}} \frac{r^{n-1}dr}{r^n}  \leq C' N. $$
Hence, from \eqref{eq.co2a3} we deduce
$$ C N^2 \leq  \|I_{n/2}(\phi)\|^2_{L^\infty(\R^n)} \leq C_1 \|\phi\|_{L^2(\R^n)}^2 \leq C_2 N ,$$
for any $N$ sufficiently big and this is impossible.

 This completes the proof of the Theorem.
\end{proof}

\section{Functional calculus kernels and their asymptotic expansions}

The functional calculus for the perturbed Hamiltonian $\mathcal{H}$ is based on the relations \eqref{eq.IFC1} and \eqref{eq.IFC1i}.
In the low energy domain we have the kernel expansion proposed in Lemma \ref{Ber1}. We shall prove this kernel estimate below.

\begin{proof}[Proof of Lemma \ref{Ber1}]
We assume $x<y$ for determinacy and consider three cases.
\begin{equation}\label{eq.case1}
    x < 0 < y , \tag{Case A}
\end{equation}
\begin{equation}\label{eq.case2}
    0 \leq x <  y,   \tag{Case B}
\end{equation}
\begin{equation}\label{eq.case3}
     x <  y \leq 0. \tag{Case C}
\end{equation}

In the \ref{eq.case1}, we can use
the representation
$$  T(\tau) m_+(y,\tau) m_-(x,\tau)= T(\tau) + T(\tau)\underbrace{ m^{rem,+}_0(y,\tau)}_{=a_1(y,\tau)}+ $$ $$ + T(\tau)\underbrace{  m^{rem,-}_0(x,\tau)}_{=a_2(x,\tau)} +  T(\tau)\underbrace{  m^{rem,+}_0(y,\tau)  m^{rem,-}_0(x,\tau)}_{= a_3(x,y,\tau)},$$
where
\begin{equation}\label{eq.Os2}
    m^{rem, \pm}_0(x,\tau) = m_\pm (x,\tau) -1.
\end{equation}
In this way, from \eqref{eq.IFC1}, we have the representation
  \begin{equation}\label{eq.remlowf}
  \varphi \left( \frac{\sqrt{\mathcal{H}}}{M} \right)(x,y) = c\ \widehat{\varphi_M}(x-y) + c\sum_{j=1}^3 I_{M} (a_j)(x,y),
  \end{equation}
where
\begin{alignat*}{2}
    I_{M} (a)(x,y) = M \int_{\rone} \varphi\left(\tau \right) T(M\tau) a(x,y,M\tau) e^{- \im M\tau(x-y)} d\tau
\end{alignat*}
and
$$ \varphi_M(\tau) = T(\tau) \varphi \left(\frac{\tau}{M} \right).$$

We can put the term $\widehat{\varphi_M}(x-y)$ in the leading term $K_M(x,y)$ defined in \eqref{eq.pasq10}, indeed, we have
\begin{equation*}
\widehat{\varphi_M}(x-y)= \mathds{1}_{x<0}\mathds{1}_{y>0}\int_{\rone}e^{-i\tau(x-y)}\varphi\left(\frac{\tau}{M}\right)T(\tau)\,d\tau.
\end{equation*}

To estimate the terms $I_M(a_j)(x,y)$ we are going to use the following fractional integration by parts estimate\footnote{here $g$ is a compactly supported function in $C^{0,\sigma}(\rone)$ such that $0 \notin {\rm supp} g$.}
\begin{equation}\label{eqIP1}
\left| \int_\rone e^{\im \tau M\xi} g (\tau) d\tau \right| \leq \frac{C}{\langle M\xi \rangle^{\sigma}} \| g \|_{C^{0,\sigma}(\rone)}, \ \ \forall \ \sigma \in (0,1).
\end{equation}
Hence we have
\begin{equation}\label{eq.Imaj}
|I_M(a_j)(x,y)|\leq C \frac{M}{\langle M(x-y)\rangle^\sigma}\left\|\varphi(\tau)\frac{T(M\tau)}{M\tau}M\tau a_j(x,y,M\tau)\right\|_{C^{0,\sigma}(\rone)}.
\end{equation}
Then, using the estimates proved in Lemma \ref{lem:Jostk0}, combined with the following estimates for $T(\tau)$
\begin{equation*}
\left\|\frac{T(\tau)}{\tau}\right\|_{C^{0,\sigma}(\rone)} +\left|\frac{T(\tau)}{\tau}\right|\leq C, \ \sigma\in(0,1)\cap(0,\gamma-1]
\end{equation*}
and with the fact that $\varphi \in C^\infty _0({\mathbf R} _+),$ such that ${\rm supp} \varphi \subseteq [1/2,2],$
we get
\begin{equation*}
\left\|\varphi(\tau)\frac{T(M\tau)}{M\tau}M\tau a_1(x,y,M\tau)\right\|_{C^{0,\sigma}(\rone)}\leq C\left(\frac{1}{\langle y \rangle^{\gamma}} +\frac{M^\sigma}{\langle y \rangle^{\gamma-\sigma}}\right),
\end{equation*}
\begin{equation*}
\left\|\varphi(\tau)\frac{T(M\tau)}{M\tau}M\tau a_2(x,y,M\tau)\right\|_{C^{0,\sigma}(\rone)}\leq C\left(\frac{1}{\langle x \rangle^{\gamma}} +\frac{M^\sigma}{\langle x \rangle^{\gamma-\sigma}}\right),
\end{equation*}
\begin{equation*}
\left\|\varphi(\tau)\frac{T(M\tau)}{M\tau}M\tau a_3(x,y,M\tau)\right\|_{C^{0,\sigma}(\rone)}\leq C\left(\frac{1}{\langle x \rangle^{\gamma-1}\langle y \rangle^{\gamma}} +\frac{1}{\langle y \rangle^{\gamma-\sigma}\langle x \rangle^{\gamma-1}}+\frac{M^{\sigma}}{\langle y \rangle^{\gamma}\langle x \rangle^{\gamma-\sigma-1}}\right).
\end{equation*}

Turning back to \eqref{eq.remlowf} and using the estimates \eqref{eq.Imaj} together with Holder estimates above, we obtain
\begin{equation*}
\left|\varphi\left(\frac{\sqrt{\mathcal{H}}}{M}\right)-c\widehat{\varphi_M}(x-y)\right|\leq C\frac{M}{\langle M(x-y)\rangle^\sigma}\left(\frac{1}{\langle y \rangle^{\gamma-\sigma}}+\frac{1}{\langle x \rangle^{\gamma-\sigma}}\right)
\end{equation*}
with $\sigma\in(0,1)\cap(0,\gamma-1]$, and $x<0<y$, i.e. we get \eqref{eq.Ber2} in the \eqref{eq.case1}.

In the \ref{eq.case2}, since we have $x\geq 0$, we want to write $m_-(x,\tau)$ in term of $m_+(x,\pm \tau)$. In order to do this, we can use the relation
\begin{equation}  \label{eq:Ber7} \begin{aligned} &    T(\tau )m_- (x ,\tau )= R_+ (\tau )e^{ 2\im \tau x }m_+ (x,\tau )+   m_+(x,-\tau ).
\end{aligned}
\end{equation}
Then we can  write
\begin{alignat*}{2}
   T(\tau) m_+(y,\tau) m_-(x,\tau) & =  \\
\left( R_+ (\tau ) + 1 \right) & e^{ 2\im \tau x } m_+(y,\tau) m_+ (x,\tau ) - \\
- & e^{ 2\im \tau x } m_+(y,\tau) m_+ (x,\tau ) + m_+(y,\tau) m_+(x,-\tau ).
\end{alignat*}
Using the remainders introduced in \eqref{eq.Os2} we can represent the kernel $ \varphi \left( \frac{\sqrt{\mathcal{H}}}{M} \right)(x,y) $ as a sum of kernels of three types:
\begin{alignat*}{2}
   & I_M(x,y) = \int_\rone \varphi\left( \frac{\tau}{M}\right) \left( \left( R_+ (\tau ) + 1 \right) e^{ 2\im \tau x } m_+(y,\tau) m_+ (x,\tau ) \right) e^{- \im \tau(x-y)} d\tau, \\ & II_M(x,y) =  M \widehat{\varphi} \left(M (x-y) \right)- M \widehat{\varphi} \left(M (x+y) \right)\\
    &III_M(x,y)=\sum_{j=1}^2 K_j(x,y;M),
\end{alignat*}
where
\begin{align}\label{eq.Ber2b}
  K_1(x,y;M)  = M \int_\R e^{-\im M \tau(x-y)} \varphi(\tau) b_1(x,y,M\tau) d\tau,
\end{align}
\begin{align}\label{eq.Ber2c}
  K_2(x,y;M)  =  -M \int_\R e^{\im M \tau(x+y)} \varphi(\tau) b_2(x,y, M\tau) d\tau,
\end{align}
and
\begin{equation*}
b_1(x,y,M\tau)=m_0^{rem,+}(y,M\tau)+m_0^{rem,+}(x,-M\tau)+m_0^{rem,+}(y,M\tau)m_0^{rem,+}(x,-M\tau),
\end{equation*}
\begin{equation*}
b_2(x,y,M\tau)=m_0^{rem,+}(y,M\tau)+m_0^{rem,+}(x,M\tau)+m_0^{rem,+}(y,M\tau)m_0^{rem,+}(x,M\tau).
\end{equation*}
As before, we firstly estimate the terms $I_M(x,y)$ and $III_M(x,y)$ with the fractional integration by parts estimate \eqref{eqIP1} and then we use Lemma \ref{lem:Jostk0} combined with the estimates
\begin{equation*}
\left\|\frac{R_\pm(\tau)+1}{\tau}\right\|_{C^{0,\sigma}(\mathbb{C}_\pm)}+\left|\frac{R_\pm(\tau)+1}{\tau}\right|\leq C, \ \sigma\in(0,1)\cap(0,\gamma-1],
\end{equation*}
and the properties of the function $\varphi$ to prove \eqref{eq.Ber2} in the \eqref{eq.case2}. Here
\begin{equation*}
b(x,y,\tau )= (R_+(\tau)+1) e^{2i\tau x}-e^{2i\tau x}+1
\end{equation*}
and $0\leq x<y$.

In the \ref{eq.case3} we follow the argument used in the \ref{eq.case2}, but this time we replace \eqref{eq:Ber7} by
\begin{equation}  \label{eq:Ber11} \begin{aligned} &    T(\tau )m_+ (y ,\tau )= R_- (\tau )e^{ -2\im \tau y }m_- (y,\tau )+   m_-(y,-\tau ),
\end{aligned}
\end{equation}
and we derive \eqref{eq.Ber2} using the argument of case \ref{eq.case2} .

This completes the proof of \eqref{eq.Ber2}.
\end{proof}

\begin{proof}[Proof of Lemma \ref{Ber1m}]In the high energy domain $M >1$ we can follow the proof of Lemma \ref{Ber1}. Using the estimates
$$ T(\tau) = 1 + O(\tau^{-1}), \ R(\tau) =O(\tau^{-1}) $$
near $ \tau \to \infty$, we can absorb the factor $M>1$ that appears in
\begin{gather*}
\varphi \left( \frac{\sqrt{\mathcal{H}}}{M} \right)(x,y) -\varphi \left(\frac{\sqrt{\mathcal{H}_0}}{M} \right)(x,y)=\\
=M\int_{\rone}\varphi(\tau)\left[T(\tau M)m_+(y,\tau M)m_-(x,\tau M)-1\right]e^{-i\tau M(x-y)}\,d\tau.
\end{gather*}
Then, proceeding as in the proof of Lemma \ref{Ber1} we  obtain the following estimate
\begin{align*}
   \left| \varphi \left( \frac{\sqrt{\mathcal{H}}}{M} \right)(x,y) -\varphi \left(\frac{\sqrt{\mathcal{H}_0}}{M} \right)(x,y)\right| \leq  \\ \nonumber \leq C  \left(\sum_{\pm} \frac{1}{\langle M(x\pm y) \rangle^{ \sigma}}\right)\left(\frac{1}{\langle x \rangle^{\gamma-\sigma}} + \frac{1}{\langle y \rangle^{\gamma-\sigma}}\right).
\end{align*}
i.e. the inequality \eqref{eq.Ber2m}.
\end{proof}

\section{ Equivalence of homogeneous Besov norms }

The comparison of homogeneous Besov spaces $\dot{B}^s_{p}(\rone)$ and $\dot{B}^s_{p,\mathcal{H}}(\rone)$
is closely connected with the definition and properties of fractional power of the Hamiltonians $\mathcal{H}$ and $\mathcal{H}_0.$

For  sectorial operators $A$ in $L^p(\rone) $ with spectrum $\sigma(A)$ satisfying
 $$ z \in \sigma(A) \setminus \{0\} \Longrightarrow {\rm Re}  z \geq c|{\rm Im}  z|, \  c>0$$
 we can define  for any $\sigma \in (0,1)$ the fractional  negative powers of $A$ as follows
  (see Theorem 1.4.2 in \cite{H})
\begin{equation}\label{eq.RRH1n}
 A^{-\sigma} = \frac{\sin (\pi \sigma)}{\pi} \int_0^\infty \lambda^{-\sigma } (\lambda +  A)^{-1} d\lambda.
\end{equation}
The above relation  suggests to introduce the fractional powers $\mathcal{H}^{s/2}$ for $s \in [0,1)$  by the aid of the relation
\begin{equation}\label{eq.RRH1}
   \mathcal{H}^{s/2} = \mathcal{H}^{-\sigma}\mathcal{H} = \frac{\sin (\pi \sigma)}{\pi} \int_0^\infty \lambda^{-\sigma } \mathcal{H} (\lambda +  \mathcal{H})^{-1} d\lambda
\end{equation}
with $\sigma=1-s/2.$

Sectorial properties of $\mathcal{H}$ are studied in \cite{GG} under the assumption that $0$ is not resonance for $\mathcal{H}.$ The convergence of the integral in \eqref{eq.RRH1} near $\theta=0$  needs justification based on limiting absorption type estimates, obtained in  \cite{GG} in the case $0$ is not a resonance point for $\mathcal{H}$.

The existence of the wave operators $W_\pm$ is well known according to the results in
\cite{W0}, \cite{AY}, \cite{DF06}, so $W_\pm$ are well  defined operators in $L^p(\rone),$ $1<p<\infty.$

The splitting property
$$ \mathcal{H }W_\pm = W_\pm \mathcal{H}_0$$
implies
that
$$ W_\pm :  \dot{B}^s_p(\rone) \to \dot{B}^s_{p,\mathcal{H}}(\rone), \ \forall s \geq 0, \ 1 < p < \infty.$$

The functional calculus for the perturbed operator $\mathcal{H}$
can be defined as follows
$$ g(\mathcal{H}) = W_+ g(\mathcal{H}_0)W_+^*= W_- g(\mathcal{H}_0)W_-^*$$
for any function $g \in L^\infty_{loc}(\rone).$

 Using the existence of the wave operators, its boundness in $L^p(\rone)$ and the splitting property, one can derive the following $L^p$ estimate (partial case of Bernstein inequality)
 \begin{equation}\label{eq.EB1}
  \left\| \varphi \left(
\frac{ \sqrt{\mathcal{H}}}{M} \right) f\right\|_{L^p(\rone)} \leq C  \|f\|_{L^p(\rone)}
 \end{equation}
 for $M>0$ and $f \in S(\rone).$

For completeness we can also mention that from  Lemma \ref{Ber1m} combined with Young convolution inequality we get the Bernstain inequality for $M>1$
 \begin{equation}\label{eq.bern}
 \left\| \varphi \left(
 \frac{ \sqrt{\mathcal{H}}}{M} \right) f-\varphi \left(
 \frac{ \sqrt{\mathcal{H}_0}}{M} \right)f\right\|_{L^p(\rone)} \leq C M^{1/p-1/q-\delta} \|f\|_{L^p(\rone)},
 \end{equation}
 where $1\leq p\leq q \leq \infty$ and $\delta>0$, $\delta= s-\sigma$ according with the notations used in Lemma \ref{Ber1m}.

 For the low energy domain, $0<M\leq 1$, we need the following estimate.

\begin{lem} \label{6.1PL} Assume that $V \in L^1_\gamma(\rone)$ with
$$ \gamma >1+1/p, \  0 <s < \frac{1}{p}, \ 1 < p < \infty.$$
Then for any even function $\varphi(\tau) \in C_0^\infty(\rone \setminus 0)$
 there exists a constant
$C =C(\| V \| _{L^{1}_{\gamma}(\rone)})$ so that for any pair of real positive  numbers $\Lambda$, $M $ such that $ 0 < \Lambda  \leq M$, $M  \leq 1$ and
 for any  $f\in \mathcal{S}( \mathbb{{R}}),$
  the following inequality holds:
\begin{equation}\label{6.3}
\left\| \varphi \left(\frac{ \sqrt{\mathcal{H}}}{M} \right)  \varphi \left(\frac{ \sqrt{\mathcal{H}_0}}{\Lambda} \right) f \right\|_{L^p_x(\rone)}
\leq C  \frac{\Lambda^{s}}{M^{s}} \|f\|_{L^p_x(\rone)}.
\end{equation}
\end{lem}

\begin{proof}

We can assume that the support of $\varphi $ is in $[1/2,2].$

We note that if $ \Lambda/4 \leq M \leq 4 \Lambda,$ or $M/4\leq \Lambda\leq 4 M$ then we can use the fact that $\varphi(\sqrt{\mathcal{H}}/M) $ and $\varphi(\sqrt{\mathcal{H}_0}/\Lambda) $ are $L^p$ bounded operators, so the estimates \eqref{6.3} is obvious in this case.

Since if $V\in L^1_\gamma(\rone)$ with $\gamma>1+1/p$ then $V\in L^1_{1+s}(\rone)$ for any $s \in [0,1/p)$.

Our first step is the  proof of \eqref{6.3}, assuming
\begin{equation}\label{eq.pasq1}
    \Lambda < M/4.
\end{equation}
Our goal is to check the inequality
\begin{equation}\label{eq.BIi1}
    \left\| \int_{\R } \varphi \left(\frac{ \sqrt{\mathcal{H}}}{M} \right) (x,y) f_\Lambda(y)  dy  \right\|_{L^p(\rone)} \leq
    C\left(\frac{\Lambda}{M}\right)^{s}   \|f\|_{L^p(\rone)}
\end{equation}
where $f_\Lambda=\varphi(\sqrt{\mathcal{H}_0}/\Lambda)f$.

We can apply the  kernel estimate \eqref{eq.Ber2} from Lemma \ref{Ber1} so we get
\begin{align}\label{EBN1}
  & \left\| \int_{ \R } \left[ \varphi \left(\frac{ \sqrt{\mathcal{H}}}{M} \right) (x,y)- K_M(x,y) \right] f_\Lambda(y)  dy  \right\|_{L^p(\rone)} \leq \nonumber\\
\leq & C M  \left\| \int_{ \R } \frac{|f_\Lambda(y)|   dy}{\langle M(x-y)\rangle^{\sigma} \langle x\rangle^{1+s-\sigma}} \right\|_{L^p(\rone)}  +  C M\left\| \int_{ \rone } \frac{|f_\Lambda(y)|   dy}{\langle M(x-y)\rangle^{\sigma} \langle y\rangle^{1+s-\sigma}} \right\|_{L^p(\rone)} ,
\end{align}
for any $\sigma \in (0,s].$

The terms in the right side of the inequality above can be evaluated using Hardy-Sobolev estimates. To be more precise, the equivalence between the Lebesgue spaces $L^p(\rone)$ and the Lorentz ones $L^{p,p}(\rone)$ in the case $1<p<\infty$ allows us to use the sharp inequalities in Lorentz spaces. Indeed, we recall that  $|x|^{-1/\beta}\in L^{\beta,\infty}(\rone)$, for any $\beta\geq 1$. Hence, using the relation
\begin{equation}\label{eq.Y5}
1+\frac{1}{p}= \sigma +(1+s-\sigma) + \left(\frac{1}{p}-s\right),
\end{equation}
 we are in position to apply Young and H\"older inequalities in Lorentz spaces to get
\begin{equation*}\label{eq.Y2}
   \left\| \int_{ \rone } \frac{|f_\Lambda(y)|   dy}{|M(x-y)|^{ \sigma} |y|^{1+s-\sigma}} \right\|_{L^{p,p}(\rone)}\leq C \frac{1}{M^{\sigma}} \left\| \frac{1}{| x |^{\sigma}}\right\|_{L^{\frac{1}{\sigma},\infty}(\rone)}\left\|\frac{1}{|y|^{1+s-\sigma}} \right\|_{L^{\frac{1}{1+s-\sigma},\infty}(\rone)}\|f_\Lambda\|_{L^{q,p}(\rone)} .
\end{equation*}
 where
\begin{equation*}
\frac{1}{q}=\frac{1}{p}-s.
\end{equation*}

This estimate can be combined with the Sobolev embedding in Lorentz spaces
\begin{equation}\label{eq.Y6}
  \|f\|_{L^{q,p}(\rone)} \leq C \|D^s f\|_{L^{p,p}(\rone)} , \ \frac{1}{q}= \frac{1}{p} - s, \ 0<s<1/p,
\end{equation}
so that we obtain that the second term in the right side in \eqref{EBN1} is bounded from
\begin{equation*}
C M^{1-\sigma}\Lambda^s \|f\|_{L^p(\rone)}\leq C\Lambda^s\|f\|_{L^p(\rone)}.
\end{equation*}

One can proceed similarly to find
\begin{equation*}
 \left\| \int_{ \R } \frac{|f_\Lambda(y)|   dy}{|M(x-y)|^{\sigma}| x|^{1+s-\sigma}} \right\|_{L^p(\rone)}\leq C\frac{1}{M^{\sigma}} \Lambda^s\|f\|_{L^p(\rone)}.
\end{equation*}
Hence we have proved that in the case $0<\Lambda<M\leq 1$
\begin{gather}\label{EBN2}
M \left\| \int_{ \R } \frac{|f_\Lambda(y)|   dy}{|M(x-y)|^{\sigma} |x|^{1+s-\sigma}} \right\|_{L^p(\rone)}  +  M\left\| \int_{ \rone } \frac{|f_\Lambda(y)|   dy}{|M(x-y)|^{ \sigma} | y|^{1+s-\sigma}} \right\|_{L^p(\rone)}\leq \\
 \leq  C \frac{M}{M^{\sigma}}\Lambda^s \left\| f\right\|_{L^p(\rone)},\nonumber
\end{gather}
with $0<\sigma\leq s$. So we have established the \eqref{6.3}.

Now we need to estimate the leading terms
\begin{equation}\label{eq.lead}
A_{M,\Lambda}(f)(x)=\int_\rone K_M(x,y) f_\Lambda(y) dy,
\end{equation}
caracterized in \eqref{eq.pasq10}.
We start with the study of the kernel
\begin{equation}\label{eq.pasq12}
K_M(x,y)= \mathds{1}_{ x>0}\mathds{1}_{ y>0} \int e^{-\im \tau(x-y)} \varphi\left(\frac{\tau}{M}\right)  \,d\tau.
\end{equation}
At first we look for the kernel $\tilde{K}_{M,\Lambda}(x,y)$, such that
\begin{equation*}
A_{M,\Lambda}(f)(x)=\int \tilde{K}_{M,\Lambda}(x,y) f(y)\,dy
\end{equation*}
and then we will find suitable bounds for $|\tilde{K}_{M,\Lambda}(x,y)|$ in order to estimate $\|A_{M,\Lambda}(f)\|_{L^p(\rone)}$. We can neglect the characteristic function $\mathds{1}_{ x>0}$. On the other side, the presence of $ \mathds{1}_{ y>0}$ and the integration in $dy$ imply that
$$ A_{M,\Lambda}(f)(x)= \iint  e^{-\im \tau x} \varphi\left(\frac{\tau}{M}\right)\frac{1}{\tau-\xi}\varphi\left(\frac{\xi}{\Lambda}\right)\hat{f}(\xi)\,d\xi\,d\tau.$$
We note that $\tau-\xi\neq 0$ since we are considering the case $\Lambda <M/4$.

By the definition of Fourier transform
\begin{equation*}
\hat{f}(\xi)=\int e^{-iy\xi}f(y)\,dy,
\end{equation*}
we get the expression of the kernel
\begin{equation}\label{eq.KMLest}
\tilde{K}_{M,\Lambda}(x,y)= \iint  e^{-\im \tau x}e^{-iy\xi} \varphi\left(\frac{\tau}{M}\right)\varphi\left(\frac{\xi}{\Lambda}\right)\frac{1}{\tau-\xi}\,d\xi\,d\tau.
\end{equation}
Operating the change of variables $\tau\mapsto M\tau$ and $\xi\mapsto \Lambda \xi$ we obtain
\begin{equation*}
\tilde{K}_{M,\Lambda}(x,y)= M\Lambda\iint  e^{-\im \tau M x}e^{-iy\Lambda\xi} \varphi\left(\tau\right)\varphi\left(\xi\right)\frac{1}{M\tau-\Lambda\xi}\,d\xi\,d\tau.
\end{equation*}
Integrating two times by parts in $\tau$ and then in $\xi$ we find the following estimate
\begin{align}\label{eq.leadest}
|\tilde{K}_{M,\Lambda}(x,y)|& \leq  \frac{M\Lambda}{\langle Mx\rangle^2\langle \Lambda y\rangle^2}\iint   \left|\partial_\xi^2\partial_\tau^2\left(\frac{\varphi\left(\tau\right)\varphi\left(\xi\right)}{M\tau-\Lambda\xi}\right)\right|\,d\xi\,d\tau\\
&\leq C\frac{\Lambda M}{\langle Mx\rangle^2\langle \Lambda y\rangle^2}\frac{1}{\max{(M,\Lambda)}}.\nonumber
\end{align}
Now we can apply H\"older inequality to get
\begin{equation}\label{eq.hold}
\|A_{M,\Lambda}(f)\|_{L^p(\rone)}\leq \frac{\Lambda}{M^{1/p}\Lambda^{1-1/p}}\|f\|_{L^p(\rone)}.
\end{equation}
We can proceed similarly for the kernels
\begin{equation}\label{eq.pasq12}
K_M(x,y)= \mathds{1}_{ x<0}\mathds{1}_{ y>0} \int e^{-\im \tau(x-y)} \varphi\left(\frac{\tau}{M}\right) T(\tau) \,d\tau
\end{equation}
and
\begin{equation}\label{eq.pasq12}
K_M(x,y)= \mathds{1}_{ \pm x>0}\mathds{1}_{ \pm y>0} \int e^{-\im \tau(x\pm y)} \varphi\left(\frac{\tau}{M}\right) (R_\pm(\tau)+1)\,d\tau.
\end{equation}
using the assumption $T(\tau) \sim \tau $, $(R_\pm(\tau)+1)\sim \tau$ near $\tau=0$ and fractional integration by parts.

Indeed, from the Theorem 2.3 in \cite{W} we have that $T(\tau)$ is $C^1(\rone)$ and $R_\pm(\tau)\in C^{0,\alpha}(\rone)$ with $\alpha <\gamma-1$. Applying $\alpha$ integration by parts we have that
\begin{equation*}
|\tilde{K}_{M,\Lambda}(x,y)|\leq C\frac{\Lambda M}{\langle Mx \rangle^\alpha\langle \Lambda y \rangle^\alpha  }\frac{1}{\max(M,\Lambda)}
\end{equation*}
and
\begin{equation*}
\|A_{M,\Lambda}(f)\|_{L^p(\rone)}\leq C \frac{\Lambda}{M^{1/p}\Lambda^{1-1/p}}\|f\|_{L^p(\rone)}
\end{equation*}
where we have chosen $\alpha>1/p$ thanks to the hypothesis $\gamma>1+1/p$.

In conclusion the estimate \eqref{6.3} is checked and it holds whenever $0 < \Lambda \leq M \leq 1.$
\end{proof}

Our proof of the equivalence of the high energy part of the homogeneous Besov norms \eqref{eq.IMi} for the perturbed Hamiltonian and the corresponding unperturbed homogeneous Besov norms is based also on the estimate of the operator
$$  \varphi \left(\frac{ \sqrt{\mathcal{H}}}{M} \right)  \varphi \left(\frac{ \sqrt{\mathcal{H}_0}}{\Lambda} \right).$$

More precisely, we have the following estimates.

\begin{lem} \label{l.ebn1} Assume that $V \in L^1_\gamma(\rone)$, $\gamma> 1+1/p$, the operator $\mathcal{H}$ has no point spectrum and resonance at zero.
Then for any even function $\varphi(\tau) \in C_0^\infty(\rone \setminus 0)$
 there exists a constant
$C =C(\| V \| _{L^{1}_{\gamma}(\rone)})$ so that for any pair of real positive  numbers $\Lambda$, $M $ and
 for any  $f\in \mathcal{S}( \mathbb{{R}}),$
  the following inequalities hold:
\begin{equation}\label{6.3m}
\left\| \varphi \left(\frac{ \sqrt{\mathcal{H}}}{M} \right)  \varphi \left(\frac{ \sqrt{\mathcal{H}_0}}{\Lambda} \right) f \right\|_{L^p_x(\rone)}
\leq C  \left(\frac{\Lambda}{M}\right)^{1/p} \|f\|_{L^p_x(\rone)}, \  \forall \  0 < \Lambda  \leq M, \ M  \geq 1
\end{equation}
and
\begin{equation}\label{6.4m}
\left\| \varphi \left(\frac{ \sqrt{\mathcal{H}}}{M} \right)  \varphi \left(\frac{ \sqrt{\mathcal{H}_0}}{\Lambda} \right) f \right\|_{L^p_x(\rone)}
\leq C  \left(\frac{M}{\Lambda}\right)^{1/p}  \|f\|_{L^p_x(\rone)},  \ \forall \  \Lambda \geq M, \ M \geq 1,
\end{equation}
with $1<p<\infty$.
\end{lem}
\begin{proof}
We shall prove first \eqref{6.3m}.
Take  $p \in (1,\infty)$ and $f,g \in S(\R).$
Set $$f_\Lambda(x)= \varphi \left(\frac{ \sqrt{\mathcal{H}_0}}{\Lambda} \right) f.$$
We use the relation
$$ \varphi \left(\frac{ \sqrt{\mathcal{H}}}{M} \right)  f_\Lambda=
 M^{-s}\varphi_1 \left(\frac{ \sqrt{\mathcal{H}}}{M} \right)  \mathcal{H}_0^{s/2} f_\Lambda +  M^{-s}\varphi_1 \left(\frac{ \sqrt{\mathcal{H}}}{M} \right) \left(\mathcal{H}^{s/2} - \mathcal{H}_0^{s/2} \right) f_\Lambda, $$
where $s>0$ will be chosen later on and
\begin{equation}\label{eq.PL10m}
    \varphi_1 \left(\tau \right) = \varphi \left(\tau \right) \tau^{-s}.
\end{equation}

Hence we have the representation formula
\begin{alignat}{2}\label{eq.PL10a1m}
   \varphi \left(\frac{ \sqrt{\mathcal{H}}}{M} \right)  f_\Lambda
=M^{-s}\varphi_1 \left(\frac{ \sqrt{\mathcal{H}}}{M} \right)&  (\mathcal{H}_0)^{s/2} f_\Lambda + G_{M,\Lambda}(f),
\end{alignat}
where
\begin{equation}\label{eq.PL10a2m}
   G_{M,\Lambda} = M^{-s}\varphi_1 \left(\frac{ \sqrt{\mathcal{H}}}{M} \right) \left(\mathcal{H}^{s/2} - \mathcal{H}_0^{s/2} \right)\varphi \left(\frac{ \sqrt{\mathcal{H}_0}}{\Lambda} \right).
\end{equation}

By  \eqref{eq.EB1}, we can write
$$ \left\| M^{-s}\varphi_1 \left(\frac{ \sqrt{\mathcal{H}}}{M} \right)  \mathcal{H}_0^{s/2} f_\Lambda    \right\|_{L^p(\rone)} \leq \frac{C}{M^s}
\left\|  \mathcal{H}_0^{s/2} f_\Lambda    \right\|_{L^p(\rone)} \leq \frac{C\Lambda^s}{M^s} \|f\|_{L^p(\rone)},
$$
so we have the estimate
\begin{equation}\label{eq.PL15m}
     \left\| M^{-s}\varphi_1 \left(\frac{ \sqrt{\mathcal{H}}}{M} \right) \mathcal{H}_0^{s/2} f_\Lambda    \right\|_{L^p(\rone)} \leq
\frac{C\Lambda^{s}}{M^{s}} \|f\|_{L^p(\rone)}.
\end{equation}

The operator $ \mathcal{H}^{s/2}-\mathcal{H}_0^{s/2}, $  entering  in the right side of \eqref{eq.PL10a2m} can be substituted by
$$
   C \int_0^\infty \lambda^{-1+s/2 } \left[\mathcal{H} (\lambda +  \mathcal{H})^{-1} - \mathcal{H}_0 (\lambda +  \mathcal{H}_0)^{-1} \right] d\lambda =
$$ \
$$ = C \int_0^\infty \lambda^{s/2 } (\lambda +\mathcal{H})^{-1} V (\lambda+ \mathcal{H}_0)^{-1}  d\lambda$$
due to \eqref{eq.RRH1}. Hence
\begin{equation}\label{eq.PL22}
   \|G_{M,\Lambda}(f)\|_{L^p(\rone)} \leq  \frac{C}{M^{s}}\int_0^\infty \lambda^{s/2} h(\lambda,\Lambda,M) d\lambda,
\end{equation}
where
$$h(\lambda, \Lambda,M) =   \left\|\varphi_1 \left(\frac{ \sqrt{\mathcal{H}}}{M} \right)  (\lambda+ \mathcal{H})^{-1} V (\lambda + \mathcal{H}_0)^{-1}  f_\Lambda \right\|_{L^p(\rone)}.$$

By \eqref{eq.bern} and the standard
estimate
\begin{equation}\label{eq.C3}
    \left\| \varphi \left( \frac{\sqrt{\mathcal{H}_0}}{\Lambda}\right) (\lambda +\mathcal{H}_0)^{-1} f \right\|_{L^q(\rone)} \leq \frac{C \Lambda^{1/p-1/q}}{\lambda + \Lambda^2} \|f\|_{L^p(\rone)},
\end{equation}

 we can write
$$  h(\lambda, \Lambda,M)  = \left\|\varphi_1 \left(\frac{ \sqrt{\mathcal{H}}}{M} \right)  (\lambda +\mathcal{H})^{-1} V (\lambda + \mathcal{H}_0)^{-1}  f_\Lambda \right\|_{L^p(\rone)} \leq $$ $$ \leq  \frac{C M^{1-1/p}}{\lambda + M^2} \left\| V (\lambda + \mathcal{H}_0)^{-1}  f_\Lambda \right\|_{L^1(\rone)} \leq$$
$$\leq   \frac{C M^{1-1/p}}{\lambda + M^2} \left\| (\lambda + \mathcal{H}_0)^{-1}  f_\Lambda \right\|_{L^\infty(\rone)} \leq  \frac{C M^{1-1/p}\Lambda^{1/p}}{(\lambda + M^2)(\lambda +\Lambda^2)}  \left\|f \right\|_{L^p(\rone)}$$
so we derive from \eqref{eq.PL22} the inequality
\begin{equation}\label{eq.PL24m}
   \|G_{M,\Lambda}(f)\|_{L^p(\rone)} \leq  C M^{1-1/p-s}\Lambda^{1/p}\int_0^\infty \frac{ \lambda^{s/2} d\lambda}{(\lambda + M^2)(\lambda +\Lambda^2)} \left\|f \right\|_{L^p(\rone)}.
\end{equation}
Now we can use the inequalities
\begin{equation}\label{eq.PL24as}
  \int_0^\infty \frac{ \lambda^{s/2} d\lambda}{(\lambda + M^2)(\lambda +\Lambda^2)} \leq   \int_0^\infty \frac{ \lambda^{s/2 -1} d\lambda}{(\lambda + M^2)} = C M^{s-2}
\end{equation}
and via \eqref{eq.PL24m} we find
\begin{equation}\label{eq.PL25}
   \|G_{M,\Lambda}(f)\|_{L^p(\rone)} \leq  C M^{-1-1/p}\Lambda^{1/p} \left\|f \right\|_{L^p(\rone)}.
\end{equation}
From this estimate, $M\geq 1$, the identity \eqref{eq.PL10a1m} and the inequality \eqref{eq.PL15m}, we see that taking $s=1/p$, we obtain
$$ \left\|\varphi \left(\frac{ \sqrt{\mathcal{H}}}{M} \right)  f_\Lambda \right\|_{L^p(\rone)} \leq \frac{C \Lambda^{1/p}}{M^{1/p}} \left\|f_\Lambda \right\|_{L^p(\rone)}
+ \frac{C \Lambda^{1/p}}{M^{1+ 1/p}} \left\|f \right\|_{L^p(\rone)} \leq  \frac{C \Lambda^{1/p}}{M^{1/p}} \left\|f \right\|_{L^p(\rone)}.$$

Similarly we can prove the case $\Lambda\geq M$, $M\geq 1$. Indeed we can use the relation
$$ \varphi \left(\frac{ \sqrt{\mathcal{H}}}{M} \right)  f_\Lambda=
M^{s}\varphi_1 \left(\frac{ \sqrt{\mathcal{H}}}{M} \right)  \mathcal{H}_0^{-s/2} f_\Lambda +  M^{s}\varphi_1 \left(\frac{ \sqrt{\mathcal{H}}}{M} \right) \left(\mathcal{H}^{-s/2} - \mathcal{H}_0^{-s/2} \right) f_\Lambda, $$
where
\begin{equation*}
\varphi_1 \left(\tau \right) = \varphi \left(\tau \right) \tau^{s}.
\end{equation*}
We can write $\mathcal{H}^{-s/2}-\mathcal{H}_0^{-s/2}$ and $\mathcal{H}^{-s/2}_0$ via \eqref{eq.RRH1n}. Then operating computations similar to the previous case and using $M\geq 1$, we get
\begin{equation*}
\left\|\varphi \left(\frac{ \sqrt{\mathcal{H}}}{M} \right)  f_\Lambda \right\|_{L^p(\rone)} \leq C\frac{M^s}{\Lambda^s}\|f\|_{L^p(\rone)},
\end{equation*}
for any $s\in(0,1)$. In particular it holds for $s=1/p$.
\end{proof}
\begin{cor}\label{cor:invlow}
 Assume that $V \in L^1_\gamma(\rone)$ with
 $$ \gamma >1+1/p, \  0 <s < \frac{1}{p}, \ 1 < p < \infty,$$
 and assume that the operator $\mathcal{H}$ has no point spectrum and resonance at zero.
 Then for any even function $\varphi(\tau) \in C_0^\infty(\rone \smallsetminus 0)$
 there exists a constant
 $C =C(\| V \| _{L^{1}_{\gamma}(\rone)})$ so that for any pair of real positive  numbers $\Lambda$, $M $ such that $ 0 < \Lambda  \leq M$, $M  \leq 1$ and
 for any  $f\in \mathcal{S}( \mathbb{{R}}),$
 the following inequality holds:
 \begin{equation}\label{6.31}
 \left\| \varphi \left(\frac{ \sqrt{\mathcal{H}_0}}{M} \right)  \varphi \left(\frac{ \sqrt{\mathcal{H}}}{\Lambda} \right) f \right\|_{L^p_x(\rone)}
 \leq C  \frac{\Lambda^{s}}{M^{s}} \|f\|_{L^p_x(\rone)}.
 \end{equation}	
\end{cor}
\begin{proof}
	By Lemma \ref{Ber1} we have that
	\begin{equation*}
	\varphi \left(\frac{ \sqrt{\mathcal{H}}}{\Lambda} \right) = K_\Lambda +\text{Rem}_\Lambda,
	\end{equation*}
	where the kernel $K_\Lambda(x,y)$ is defined in \eqref{eq.pasq10} and the kernel of the remainder $\text{Rem}_\Lambda(x,y)$ satisfies the estimate \eqref{eq.Ber2}.
	
	We first estimate the remainder term. By \eqref{eq.Ber2} we have
	\begin{gather*}
	 \left\| \varphi \left(\frac{ \sqrt{\mathcal{H}_0}}{M} \right)  \text{Rem}_\Lambda f \right\|_{L^p_x(\rone)}
	 \leq  \\
	 \leq C \Lambda \left\|\int_{\rone}\left(\sum_{\pm} \frac{1}{\langle \Lambda(x\pm y) \rangle^{ \sigma}}\right)\left(\frac{1}{\langle x \rangle^{\gamma-\sigma}} + \frac{1}{\langle y \rangle^{\gamma-\sigma}}\right)|f(y)|\,dy\right\|_{L^p(\rone)},
	\end{gather*}
	where $\sigma \in (0,1)\cap(0,\gamma-1]$ will be choose small enough. Applying H\"older and Young inequalities in Lorentz spaces with the following index relation
	\begin{equation*}
	\frac{1}{p}+1=\sigma +(1-\sigma)+\frac{1}{p},
	\end{equation*}
	we get
	\begin{equation*}
	\left\|\varphi\left(\frac{\sqrt{\mathcal{H}_0}}{M}\right)\text{Rem}_\Lambda f\right\|_{L^p(\rone)}\leq C\frac{\Lambda}{\Lambda^\sigma}\|f\|_{L^p(\rone)}\leq C \frac{\Lambda^s}{M^s}\|f\|_{L^p(\rone)}.
	\end{equation*}
	Now we turn to estimate the leading term. We consider
	\begin{equation*}
	K_\Lambda(x,y)= c\mathds{1}_{ x>0}\mathds{1}_{ y>0} \int_{\rone}e^{-i\tau (x-y)}\varphi\left(\frac{\tau}{\Lambda}\right)\,d\tau
	\end{equation*}
	since we can proceed similarly for the other terms defined in \eqref{eq.pasq10}.
	
	We look for the kernel $\tilde{K}_{M,\Lambda}(x,y)$ such that
	\begin{equation*}
	\varphi\left(\frac{\sqrt{\mathcal{H}_0}}{M}\right)\left(\int_{\rone}K_\Lambda(\cdot,y)f(y)\,dy\right) (x)= \int_{\rone}\tilde{K}_{M,\Lambda}(x,y)f(y)\,dy.
	\end{equation*}
	We put
	\begin{equation*}
	h(x)= \int dy \int d\tau\, e^{-i(x-y)\tau}\varphi\left(\frac{\tau}{\Lambda}\right)f(y)\mathds{1}_{y>0}
	\end{equation*}
	and
	\begin{equation*}
	g(x)= c\mathds{1}_{x>0}h(x).
	\end{equation*}
	Using the notation above we have that
	\begin{equation*}
	\varphi\left(\frac{\sqrt{\mathcal{H}_0}}{M}\right)g(x)= c\int e^{ix\xi}\varphi\left(\frac{\xi}{M}\right)\hat{g}(\xi)\,d\xi
	\end{equation*}
	and
	\begin{equation*}
	\hat{h}(\eta)=c  \varphi\left(\frac{\eta}{\Lambda}\right)\int dy  e^{iy\eta}f(y)\mathds{1}_{y>0}.
	\end{equation*}
	Hence we deduce
	\begin{equation*}
	\varphi\left(\frac{\sqrt{\mathcal{H}_0}}{M}\right)g(x)= \int dy \underbrace{\int d\xi \int d\eta e^{ix\xi} e^{iy \eta}\varphi\left(\frac{\xi}{M}\right) \varphi\left(\frac{\eta}{\Lambda}\right)\frac{1}{\xi-\eta}}_{\tilde{K}_{M,\Lambda}(x,y)} f(y)\mathds{1}_{y>0}.
	\end{equation*}
	Then, as in \eqref{eq.KMLest}, integrating by parts we get
	\begin{equation*}
	\left|\tilde{K}_{M,\Lambda}(x,y)\right|\leq C\frac{M\Lambda}{\langle Mx \rangle^2\langle \Lambda y \rangle^2}\frac{1}{\max(M,\Lambda)}.
	\end{equation*}
	We note that we are considering the $0<\Lambda<M\leq 1$. Then, using H\"older inequality combined with a scaling argument we get
	\begin{equation*}
	\left\|\varphi\left(\frac{\sqrt{\mathcal{H}_0}}{M}\right)g\right\|_{L^p(\rone)}\leq C \frac{\Lambda}{M^{1/p}\Lambda^{1-1/p}}\|f\|_{L^p(\rone)}.
	\end{equation*}
	
\end{proof}

\begin{cor}\label{cor:invhigh}
Assume that $V \in L^1_\gamma(\rone)$, $\gamma> 1+1/p$, the operator $\mathcal{H}$ has no point spectrum and resonance at zero.
Then for any even function $\varphi(\tau) \in C_0^\infty(\rone \setminus 0)$
there exists a constant
$C =C(\| V \| _{L^{1}_{\gamma}(\rone)})$ so that for any pair of real positive  numbers $\Lambda$, $M $ and
for any  $f\in \mathcal{S}( \mathbb{{R}}),$
the following inequalities hold:
\begin{equation}\label{6.3m1}
\left\| \varphi \left(\frac{ \sqrt{\mathcal{H}_0}}{M} \right)  \varphi \left(\frac{ \sqrt{\mathcal{H}}}{\Lambda} \right) f \right\|_{L^p_x(\rone)}
\leq C  \left(\frac{\Lambda}{M}\right)^{1/p} \|f\|_{L^p_x(\rone)},  \ \forall \  0 < \Lambda  \leq M, \ M  \geq 1
\end{equation}
and
\begin{equation}\label{6.4m1}
\left\| \varphi \left(\frac{ \sqrt{\mathcal{H}_0}}{M} \right)  \varphi \left(\frac{ \sqrt{\mathcal{H}}}{\Lambda} \right) f \right\|_{L^p_x(\rone)}
\leq C  \left(\frac{M}{\Lambda}\right)^{1/p}  \|f\|_{L^p_x(\rone)},  \ \forall \  \Lambda \geq M, \ M \geq 1,
\end{equation}
with $1<p<\infty$.		
\end{cor}
\begin{proof}
	The proof of the inequalities \eqref{6.3m1} and \eqref{6.4m1} follows repeating the same arguments of Lemma \ref{l.ebn1} and replacing $\mathcal{H}$ with $\mathcal{H}_0$ and vice versa.
\end{proof}

Now we can turn to the following.
\begin{proof}[Proof  of Theorem \ref{MT1}]
	We have to prove the equivalence of the norms in \eqref{eq:BS1} and \eqref{eq:BS2}, i.e.
	\begin{equation}\label{eq.eqn1}
	\sum_{k=-\infty}^\infty 2^{2ks} \left\| \varphi \left(
	\frac{ \sqrt{\mathcal{H}_V}}{2^k} \right)f\right\|^2_{L^p(\R)} \sim  \sum_{j=-\infty}^\infty 2^{2js} \left\| \varphi \left(
	\frac{ \sqrt{\mathcal{H}_0}}{2^j} \right)f\right\|^2_{L^p(\R)}.
	\end{equation}
	We set
	$$ a_k =  \left\| \varphi \left(
	\frac{ \sqrt{\mathcal{H}}}{2^k} \right)f\right\|_{L^p(\R)}, \ b_j =\left\| \varphi \left(
	\frac{ \sqrt{\mathcal{H}_0}}{2^j} \right)f\right\|_{L^p(\R)}.$$
	Using the Paley-Littlewood partition
	$$ f = \sum_{j=-\infty}^\infty f_j = \sum_{j=-\infty}^\infty \varphi \left(
	\frac{ \sqrt{\mathcal{H}_0}}{2^j} \right)f,$$
	we take $\psi(\tau) \in C_0^\infty ({\mathbf R} _+) $ such that $\psi(\tau) =1$ on the support of $\varphi.$ Then we can use the identity
	\begin{equation}\label{eq.paley}
	\varphi \left(
	\frac{ \sqrt{\mathcal{H}_V}}{2^k} \right)f = \sum_{j=-\infty}^\infty \varphi \left(
	\frac{ \sqrt{\mathcal{H}_V}}{2^k} \right) \psi \left(
	\frac{ \sqrt{\mathcal{H}_0}}{2^j} \right) f_j.
	\end{equation}
	
	We distinguish the two cases $k\geq 0$ and $k<0$.
	
Let $k\geq 0$ be fixed. We can apply Lemma \ref{l.ebn1} and we obtain that
	\begin{equation*}
	a_k = \left\| \varphi \left(
	\frac{ \sqrt{\mathcal{H}_V}}{2^k} \right)f\right\|_{L^p(\R)} \leq C \sum_{j=-\infty}^\infty  2^{-|k-j|(1/p)} \left\| f_j\right\|_{L^p(\R)} =C
	\sum_{j=-\infty}^\infty  2^{-|k-j|(1/p)} b_j.
	\end{equation*}
	From this we deduce that
	\begin{equation}\label{eq.halfeq}
	\left\| 2^{ks}a_k\right\|_{\ell^2_{k\geq 0}}\leq C \left\|2^{js}b_j\right\|_{\ell^2_j(\mathbb{Z})}.
	\end{equation}
	Indeed we have
	\begin{equation}\label{eq.es1}
	\left\| 2^{ks}a_k\right\|_{\ell^2_{k\geq 0}(\mathbb{Z})}\leq C\left\| \sum\nolimits_{j\in\mathbb{Z}} 2^{-|j-k|(1/p)}2^{-(j-k)s} 2^{js}\|f_j\|_{L^p(\R)}\right\|_{\ell^2_{k\geq 0}}.
	\end{equation}
	Using the discrete Young inequality combined with
	\begin{equation}\label{eq.es2}
	\left\|2^{{-|n|(1/p)}-ns}\right\|_{\ell^1_n(\mathbb{Z})}\leq C,
	\end{equation}
	with $0<s<1/p$, we get the inequality \eqref{eq.halfeq}.
	
	Let $k<0$ be fixed. Then we write
	\begin{align*}
	2^{ks}a_k\leq 2^{ks} \sum_{j\leq k}\left\|\varphi \left(
	\frac{ \sqrt{\mathcal{H}_V}}{2^k} \right) \psi \left(
	\frac{ \sqrt{\mathcal{H}_0}}{2^j} \right) f_j\right\|_{L^p(\rone)} +\\
	+2^{ks}\sum_{j\geq k}\left\|\varphi \left(
	\frac{ \sqrt{\mathcal{H}_V}}{2^k} \right) \psi \left(
	\frac{ \sqrt{\mathcal{H}_0}}{2^j} \right) f_j\right\|_{L^p(\rone)}.
	\end{align*}
	Now we estimate the $\ell^2_{k\leq 0}$ norm of the two addends above.
	
	We can estimate the first addend as in the case $k>0$ using the inequality \eqref{6.3} and the index $s'$ such that $0<s<s'<1/p$. Then we can proceed as in \eqref{eq.es1}, \eqref{eq.es2} replacing $1/p$ with $s'$.
	
	For the second addend the estimate is simpler. Indeed, using \eqref{eq.EB1} we have
	\begin{equation}\label{eq.ban}
	2^{ks}\sum_{j\geq k}\left\|\varphi \left(
	\frac{ \sqrt{\mathcal{H}_V}}{2^k} \right) \psi \left(
	\frac{ \sqrt{\mathcal{H}_0}}{2^j} \right) f_j\right\|_{L^p(\rone)}\leq C \sum_{j\geq k} 2^{ks}2^{-js} 2^{js}\|f_j\|_{L^p(\rone)}.
	\end{equation}
	Since we are considering the case $j\geq k$ and $k<0$, we can estimate the right side above with the sum
	\begin{equation*}
	C\sum_{j\in\mathbb{Z}} 2^{-|k-j|s} 2^{js}\|f_j\|_{L^p(\rone)}.
	\end{equation*}
	Now, computing the $\ell^2_k$ norm and applying the discrete Young inequality we complete the proof of the estimate
	\begin{equation}\label{eq.halfeq}
	\left\| 2^{ks}a_k\right\|_{\ell^2_{k}(\mathbb{Z})}\leq C \left\|2^{js}b_j\right\|_{\ell^2_j(\mathbb{Z})}.
	\end{equation}
	
	To prove that
	\begin{equation}\label{eq.halfeq2}
	\left\| 2^{js}b_j\right\|_{\ell^2_j(\mathbb{Z})}\leq C \left\|2^{ks}a_k\right\|_{\ell^2_j(\mathbb{Z})},
	\end{equation}
	we use Corollary \ref{cor:invlow} and Corollary \ref{cor:invhigh}. Indeed, if we write
	\begin{equation*}
	\varphi \left(
	\frac{ \sqrt{\mathcal{H}_0}}{2^j} \right)f = \sum_{k=-\infty}^\infty \varphi \left(
	\frac{ \sqrt{\mathcal{H}_0}}{2^j} \right) \psi \left(
	\frac{ \sqrt{\mathcal{H}}}{2^k} \right) f_k,
	\end{equation*}
	where
	\begin{equation*}
	f_k= \varphi \left(
	\frac{ \sqrt{\mathcal{H}_V}}{2^k}\right)f,
	\end{equation*}
	as before we can distinguish the case $j\geq 0$ and $j<0$.

	Computations similar to the ones  used to prove \eqref{eq.halfeq} conclude the proof.
\end{proof}

\section{Estimates for the modified Jost functions}
\label{sec:spectral}

In this section we recall some classical results concerning the spectral decomposition of the perturbed Hamiltonian.
Recall that the  Jost functions are  solutions $f_{
\pm } (x,\tau )=e^{\pm i\tau x}m_{
\pm } (x,\tau )$ of $\mathcal{H }u=\tau ^2 u$ with
$$ \lim _{x\to +\infty }   {m_{ + } (x,\tau )}  =1 =
\lim _{x\to -\infty }  {m_{- } (x,\tau )}  . $$
We set  $x_+:=\max \{ 0,x \}$,   $x_-:=\max \{ 0,-x \}$.


The  estimate and  the asymptotic expansions of $m_\pm(x,\tau)$ are based on  the following integral equations
\begin{alignat}{2}\label{eq.Igra1n}
        m_\pm(x,\tau) = 1+ K_{\pm}^{(\tau)} (m_{\pm}(\cdot, \tau))(x),
    \end{alignat}
    where
    $ K_\pm^{(\tau)}$ is the integral  operator defined as follows
$$ K_\pm^{(\tau)}  (f)(x) =
   \pm  \int_x^{\pm \infty} D(\pm(t-x),\tau) V(t) f(t) dt $$
    and
\begin{equation}\label{eq.AE3}
     D(t,\tau) = \frac{e^{2it\tau}-1}{2i\tau} = \int_0^t e^{2iy\tau} dy;
\end{equation}

The following lemma  is well known.
\begin{lem} (see Lemma 1 p. 130 \cite{DeiTru})
\label{lem:Jost} Assume  $V \in L^{1}_2(\rone)$. Then we have the properties:
 \begin{enumerate}[noitemsep,label=\alph*)]
   \item for any $x \in \R$ the function
\begin{equation}\label{eq.Jo1}
   \tau \in \overline{\cone_\pm} \mapsto m_\pm (x , \tau),\ \ \cone_\pm = \{\tau \in \cone; {\rm Im} \tau \gtrless 0 \}
\end{equation}
is analytic in $\cone_\pm$ and $  C^{1} (\overline{\cone_\pm}  );$
   \item  there exist  constants $C_1$  and $C_2>0$
such that for any $x, \tau \in \rone$:
   \begin{align}  \label{eq:kernel2n}
 &  |m_\pm(x, \tau )-1|\le  C _1 \langle  x_{\mp}\rangle
\langle \tau \rangle ^{-1}  \ ;
  \\&   \label{eq:tderm}   |  \partial _\tau   m_\pm(x, \tau ) | \le  C_2 \langle x \rangle^2.
     \end{align}
 \end{enumerate}
    \end{lem}

A slight improvement is given in the next Lemma.
\begin{lem}\label{lem:Jostk0}
	Suppose $V\in L^1_\gamma(\rone)$ with $\gamma\geq 1$. Then we have the following properties:
	\begin{itemize}
		\item[a)] There exists a constant $C>0$ such that for any $x\in \rone$, $\tau\in \overline{\mathbb{C}_\pm}$, we have
		\begin{equation}\label{eq.1k0}
		\left|m_\pm(x,\tau)-1\right|\leq C\frac{\langle x_\mp\rangle}{\langle x_\pm\rangle^{\gamma-1}};
		\end{equation}
		\item[b)]There exists a constant $C>0$ such that for any $x\in \rone$, $\tau\in \overline{\mathbb{C}_\pm}\smallsetminus\{0\}$, we have
		\begin{equation}\label{eq.2k0}
		\left|m_\pm(x,\tau)-1\right|\leq C\frac{\langle x_\mp\rangle}{\langle x_\pm\rangle^{\gamma}|\tau|};
		\end{equation}
		\item[c)]Let $\sigma\in[0,1)$. Then there exists a constant $C>0$ such that for any $x\in \rone$ we have
		\begin{equation}\label{eq.3k0}
		\left\|m_\pm(x,\tau)-1\right\|_{C^{0,\sigma}(\mathbb{C}_\pm)}\leq C\frac{\langle x_\mp\rangle^{1+\sigma}}{\langle x_\pm\rangle^{\gamma-1-\sigma}}, \ \gamma>1,\  0\leq\sigma\leq \gamma-1;
		\end{equation}
	\item[d)]Let $\sigma\in[0,1)$. Then there exists a constant $C>0$ such that for any $x\in \rone$ we have
	\begin{equation}\label{eq.4k0}
	\left\|\tau (m_\pm(x,\tau)-1)\right\|_{C^{0,\sigma}(\mathbb{C}_\pm)}\leq C\frac{\langle x_\mp\rangle^{1+\sigma}}{\langle x_\pm\rangle^{\gamma-\sigma}}, \ \gamma>1.
	\end{equation}	
	\end{itemize}
\end{lem}
\begin{proof}
	We can fix for determinacy the sign $+$ in the left sides of the inequalities \eqref{eq.1k0}-\eqref{eq.4k0}, since the arguments are similar for the term $m_-$.
	
	We start proving the \eqref{eq.1k0}. The right side of \eqref{eq.1k0} suggests to consider the quantity
	\begin{equation*}
	v(x,\tau)=\frac{\langle x_+\rangle^{\gamma-1}}{\langle x_-\rangle}|m_+(x,\tau)-1|.
	\end{equation*}
	We plan to use the integral equation for $m_+(x,\tau)$ and to
	check inequality of type
	\begin{equation}\label{eq.in40b}
	v(x) \leq a(x) + \int_x^\infty b(t) v(t) dt,
	\end{equation}
	where $b \in L^1(\rone).$
	Applying for $v(x)$  a Gronwall type inequality (see  Lemma \ref{l.a2.1rone} for the precise statement), we can derive apriori bound $v(x) \leq C (a(x),\|b\|_{L^1(\rone)}).$
	
	The relations
	\begin{equation}\label{eq.intmp}
	m_+(x,\tau)-1= \int_x^{+\infty}D(t-x,\tau)V(t)m_+(t,\tau) dt,
	\end{equation}
	\begin{equation*}
	|D(t-x,\tau)|\leq C\langle t-x\rangle\leq C(\langle t\rangle+\langle x_-\rangle),
	\end{equation*}
	imply the following estimate
	\begin{equation*}
	v(x,\tau)\leq C\int_x^{+\infty}\frac{\langle x_+\rangle^{\gamma-1}}{\langle x_-\rangle}\frac{\langle t-x\rangle }{\langle t \rangle^\gamma}\langle t\rangle^\gamma|V(t)|\left(|m_+(t,\tau)-1|+1\right)\,d\tau.
	\end{equation*}
	We set\footnote{To prove that the quantity above are finite, we consider three different cases: $x<t<0$, $0<x<t$, $x<0<t$ separately. In the last case, we distinguish the behaviour if $x\approx t$, $|x|<<|t|$ and $|t|<<|x|$. }
	\begin{equation*}
	c_1=\sup_{t\geq x}\frac{\langle x_+\rangle^{\gamma-1}\langle t-x \rangle \langle t_-\rangle }{\langle x_-\rangle \langle t\rangle^\gamma \langle t_+\rangle ^{\gamma}}\in \rone_+,\ \gamma\geq 1,
	\end{equation*}
	\begin{equation*}
	c_2=\sup_{t\geq x}\frac{\langle x_+\rangle^{\gamma-1}\langle t-x \rangle }{\langle x_-\rangle \langle t\rangle^\gamma}\in \rone_+,\ \gamma\geq 1
	\end{equation*}
	and we deduce that
	\begin{equation*}
	v(x,\tau)\leq c_1\int_{x}^{+\infty}\langle t\rangle^\gamma |V(t)|v(t,\tau)\,dt+c_2\|V\|_{L^1_\gamma(\rone)}.
	\end{equation*}
	Now applying the Gronwall argument of Lemma \ref{l.a2.1rone} mentioned above, we find the \eqref{eq.1k0}.
	
	We follow the same idea to prove the other inequalities.
	
	Indeed, to get the \eqref{eq.2k0} we define
	\begin{equation}
	u(x,\tau)=|\tau|\frac{\langle x_+\rangle^\gamma}{\langle x_-\rangle }|m_+(x,\tau)-1|.
	\end{equation}
	This time we quote the estimates
	\begin{equation}\label{inv2}
	|D(t-x,\tau)|\leq C\min\left(\langle t-x \rangle ,\frac{1}{|\tau|}\right).
	\end{equation}
	Hence, by the integral equation \eqref{eq.intmp} and the estimates above follows
	\begin{align*}
	u(x,\tau)\leq& \int_x^{+\infty}\frac{\langle x_+\rangle^\gamma \langle t_-\rangle }{\langle x_-\rangle \langle t_+\rangle ^\gamma}|D(t-x,\tau)||V(t)|u(t,\tau)\,d\tau\\
	& \,\,\,+\int_x^{+\infty}\frac{\langle x_+\rangle^\gamma  }{\langle x_-\rangle  }|\tau||D(t-x,\tau)||V(t)|\,d\tau.
	\end{align*}
	As before\footnote{One can see footnote {$3$}} we can set
	\begin{equation*}
	c_1=\sup_{t\geq x}\frac{\langle x_+\rangle^{\gamma}\langle t-x \rangle \langle t_-\rangle }{\langle x_-\rangle \langle t\rangle^\gamma \langle t_+\rangle ^{\gamma}}\in \rone_+,\ \gamma\geq 1,
	\end{equation*}
	\begin{equation*}
	c_2=\sup_{t\geq x}\frac{\langle x_+\rangle^{\gamma} }{\langle x_-\rangle \langle t\rangle^\gamma}\in \rone_+,\ \gamma\geq 1
	\end{equation*}
	and via Gronwall argument we get $u(x,\tau)\leq C(\|V\|_{L^\gamma(\rone)})$, i.e. \eqref{eq.2k0}.
	
	Similarly to get the \eqref{eq.3k0} we put
	\begin{equation*}
	g^\sigma(x,\tau_1,\tau_2)=\frac{\langle x_+\rangle ^{\gamma-1-\sigma}}{\langle x_-\rangle^{\sigma+1}}\frac{|m_+(x,\tau_1)-m_+(x,\tau_2)|}{|\tau_1-\tau_2|^\sigma}
	\end{equation*}
	and by the estimate
	\begin{equation*}
	\frac{|D(t,\tau_1)-D(t,\tau_2)|}{|\tau_1-\tau_2|^\sigma}\leq C\langle t-x\rangle^{1+\sigma}\leq C(\langle t\rangle^{1+\sigma}+\langle x_-\rangle^{1+\sigma}), \ \sigma\in(0,1)
	\end{equation*}
	we get
	\begin{align*}
	g^\sigma(x,\tau_1,\tau_2)\leq& \int_x^{+\infty}\frac{\langle x_+\rangle ^{\gamma-1-\sigma}\langle t-x\rangle^{1+\sigma}}{\langle x_-\rangle^{\sigma+1}}|V(t)||m_+(t,\tau_1)|\,dt +\\
	&\,\,\, +\int_x^{+\infty}\frac{\langle x_+\rangle ^{\gamma-1-\sigma}\langle t-x\rangle \langle t_-\rangle^{\sigma+1}}{\langle x_-\rangle^{\sigma+1}\langle t_+\rangle^{\gamma-1-\sigma}}|V(t)|g^\sigma(t,\tau_1,\tau_2)\,dt.
	\end{align*}
	Moreover, we can estimate $|m_+(t,\tau_1)|$ with \eqref{eq.1k0}. \\
	If we consider $1<\gamma<2$ and $\sigma\leq \gamma-1$ or $\gamma\geq 2$ and $\sigma\in (0,1)$, we have that the following quantities are finite\footnote{One can see the footnote {${3}$}}
	\begin{equation*}
	c_1=\sup_{t\geq x}\frac{\langle x_+\rangle^{\gamma-1-\sigma}\langle t-x \rangle^{\gamma-1-\sigma}  }{\langle x_-\rangle^{1+\sigma} \langle t\rangle^\gamma }\in \rone_+,
	\end{equation*}
	\begin{equation*}
	c_2=\sup_{t\geq x}\frac{\langle x_+\rangle^{\gamma-1-\sigma}\langle t-x \rangle \langle t_-\rangle^{1+\sigma}  }{\langle x_-\rangle^{1+\sigma} \langle t\rangle^\gamma\langle t_+\rangle^{\gamma-1-\sigma} }\in \rone_+,
	\end{equation*}
	\begin{equation*}
	c_3=\sup_{t\geq x}\frac{\langle x_+\rangle^{\gamma-1-\sigma}\langle t-x \rangle^{1+\sigma} \langle t_-\rangle  }{\langle x_-\rangle^{1+\sigma} \langle t\rangle^\gamma\langle t_+\rangle^{\gamma-1} }\in \rone_+.
	\end{equation*}
	Then we have $g^\sigma(x,\tau_1,\tau_2)\leq C(\|V\|_{L^1_\gamma(\rone)})$.

Finally we prove the inequality \eqref{eq.4k0} for any $ \sigma \in (0, 1) $. We rewrite \eqref{eq.intmp} as
\begin{eqnarray}\label{eq.Dv1}
	\tau \left( m_+(x,\tau)-1\right) & = & \tau \int_x^{+\infty}D(t-x,\tau)V(t) dt + \\ \nonumber
 & + &\int_x^{+\infty}\tau D(t-x,\tau)V(t)\left( m_+(t,\tau)-1\right) dt.
	\end{eqnarray}
Setting now
 \begin{equation*}
	h^\sigma(x,\tau)= \frac{\langle x_+\rangle ^{\gamma-\sigma}}{\langle x_-\rangle^{\sigma+1}} \left\|\tau (m_+(x,\tau )-1)\right\|_{C^{0,\sigma}(\mathbb{C}_+)},
	\end{equation*}
we can use the inequality
$$ \|f g \|_{C^{0,\sigma}} \leq C \left( \|f \|_{C^{0,\sigma}} \|g \|_{C^{0}} + \|f \|_{C^{0}} \|g \|_{C^{0,\sigma}} \right) $$
and arrive at the estimate
$$ h^\sigma(x,\tau) \leq  \underbrace{ \int_x^\infty\frac{\langle x_+ \rangle^{\gamma - \sigma}}{\langle x_- \rangle^{1+\sigma}} \left\| \tau D(t-x,\tau)\right\|_{C^{0,\sigma}_\tau(\mathbb{C}_+)} |V(t)| dt}_{I(x)} + $$ $$ +\underbrace{ \int_x^\infty \frac{\langle x_+ \rangle^{\gamma - \sigma}}{\langle x_- \rangle^{1+\sigma}}\left\| \tau D(t-x,\tau)\right\|_{C^{0,\sigma}_\tau(\mathbb{C}_+)} |V(t)| \|m_+(t,\tau)-1 \|_{C^0_\tau(\mathbb{C}_+)} dt}_{II(x)}+$$
$$ + \underbrace{\frac{\langle x_+ \rangle^{\gamma - \sigma}}{\langle x_- \rangle^{1+\sigma}} \int_x^\infty \left\|  D(t-x,\tau)\right\|_{C^{0}_\tau(\mathbb{C}_+)} \frac{|V(t)| \langle t_- \rangle^{\sigma + 1} h^\sigma(t.\tau) dt}{\langle t_+ \rangle^{\gamma -\sigma}}}_{III(x)} $$
	We quote the inequalities
	\begin{equation}\label{inw1}
	\|\tau^{1-k} D(t-x,\tau)\|_{C^{0,\sigma}(\mathbb{C}_+)}\leq C\langle t-x\rangle^{k+\sigma} ,\ \ k=0,1, \ \sigma \in [0,1).
	\end{equation}
For the term $I(x)$ we use the estimate \eqref{inw1} with $k=0$ and note that
\begin{equation*}
	c_1=\sup_{t\geq x}\frac{\langle x_+\rangle^{\gamma-\sigma}\langle t-x \rangle^{\sigma}  }{\langle x_-\rangle^{1+\sigma} \langle t\rangle^\gamma }< \infty,
	\end{equation*}
for\footnote{the only case, when $\sigma \leq \gamma$ is necessary is the case $x < 0 < t$, $|x| \ll |t|$} $0 \leq \sigma \leq \gamma$, $\gamma \geq 1.$
Hence,
$$ I(x) \leq c_1 \|V\|_{L^1_\gamma(\rone)}.$$
In a similar way, for $II(x)$ we use the estimate \eqref{inw1} with $k=0$ combined with \eqref{eq.1k0} and using the estimate
\begin{equation*}
	c_2=\sup_{t\geq x}\frac{\langle x_+\rangle^{\gamma-\sigma}\langle t-x \rangle^{\sigma} \langle t_-\rangle  }{\langle x_-\rangle^{1+\sigma} \langle t\rangle^\gamma\langle t_+\rangle^{\gamma-1} }\in \rone_+,
	\end{equation*}
for $0 \leq \sigma < 1$, $\gamma \geq 1,$
we
arrive at
$$ II(x) \leq c_2\|V\|_{L^1_\gamma(\rone)}.$$
Finally, for $III(x)$ we use $\left\|  D(t-x,\tau)\right\|_{C^{0}_\tau(\mathbb{C}_+)}  \leq C\langle t-x \rangle$ and from
\begin{equation*}
	c_3=\sup_{t\geq x}\frac{\langle x_+\rangle^{\gamma-\sigma}\langle t-x \rangle^{} \langle t_-\rangle^{1+\sigma}  }{\langle x_-\rangle^{1+\sigma} \langle t\rangle^\gamma\langle t_+\rangle^{\gamma-\sigma} } < \infty,
	\end{equation*}
we deduce
$$ III(x) \leq c_3\int_x^\infty \langle t \rangle^\gamma |V(t)| h^\sigma(t,\tau) dt.$$
So, the application of the Gronwall argument implies $h^\sigma(x,\tau)\leq C$ and the estimate \eqref{eq.4k0} is established. This complete the proof.
\end{proof}
Similarly, if we require more decay for the potential, we can get estimates also for the quantity $\partial_\tau^{k}(m_\pm(x,\tau)-1)$. In particular we have the following Lemma.
\begin{lem}
\label{lem:JostIM} Suppose  $V \in L^{1}_\gamma(\rone)$ with $\gamma \geq 1$. Then  we have the following properties:
\begin{enumerate}[noitemsep,label=\alph*)]
   \item  If $\gamma\geq 2$, then for any integer $k$,  $1 \leq k \leq \gamma -1$ the function in \eqref{eq.Jo1} is  $ C^k (\overline{\cone_\pm}  )$ and there exists a constant $C>0$
such that for any $x \in \rone$ and $\tau \in  (\overline{\cone_\pm}  ) $ we have
   \begin{equation}\label{eq.JImpr01}
  \left| \partial_{\tau}^k \left(  m_\pm(x, \tau ) - 1\right) \right| \le  C \frac{ \langle x_\mp \rangle^{1+k}}{\langle x_\pm \rangle^{\gamma-1-k}};
\end{equation}
\item For any integer $k$, $1\leq k\leq \gamma,$ then the function  in \eqref{eq.Jo1} is $ C^k (\overline{\cone_\pm}  )$ and there exists a constant $C>0$
such that for any $x \in \rone$ and $\tau \in  (\overline{\cone_\pm}\smallsetminus\{0\}  ) $ we have
  \begin{equation}\label{eq.JImpr02}
  \left| \partial_{\tau}^k \left(  m_\pm(x, \tau ) - 1\right) \right| \le  C \frac{ \langle x_\mp \rangle^{1+k}}{\langle x_\pm \rangle^{\gamma-k}|\tau|};
\end{equation}
\item  If $\gamma> 2$, then for any integer $k$,  $1 \leq k \leq \gamma -1$ and for any $\sigma\in(0,1)$ such that $0\leq \sigma \leq \gamma-1-k$, there exists a constant $C>0$
such that for any $x \in \rone$ we have
\begin{equation}\label{eq.JImpr03}
\left\|  m_\pm(x, \tau ) - 1 \right\|_{C^{k,\sigma}(\mathbb{C}_\pm)} \le  C \frac{ \langle x_\mp \rangle^{1+k+\sigma}}{\langle x_\pm \rangle^{\gamma-1-k-\sigma}};
\end{equation}
\item  If $\gamma> 1$, then for any integer $k$,  $1 \leq k \leq \gamma $ and for any $\sigma\in(0,1)$ such that $0\leq \sigma \leq \gamma-k$, there exists a constant $C>0$ such that for any $x \in \rone$ we have
\begin{equation}\label{eq.JImpr04}
\left\|  \tau  (m_\pm(x, \tau ) - 1) \right\|_{C^{k,\sigma}(\mathbb{C}_\pm)} \le  C \frac{ \langle x_\mp \rangle^{1+k+\sigma}}{\langle x_\pm \rangle^{\gamma-k-\sigma}}.
\end{equation}
 \end{enumerate}
    \end{lem}

    \begin{proof} The proof of this Lemma follows the same spirit of the proof of the previous one.
    	
    	We prove the inequality \eqref{eq.JImpr01} fixing the sign $+$ in the left side. The arguments are similar for the others inequalities and also for the terms $m_-$.
    	
    	The right side in \eqref{eq.JImpr01} suggests us to define
    	$$ v^{(k)}(x) = \frac{\langle x_+ \rangle^{\gamma -1-k}}{\langle x_-\rangle^{k+1}}\left| \partial^k_\tau(m_+(x,\tau)-1)\right|.$$
    	We intend to prove the \eqref{eq.JImpr01}, i.e.
    	\begin{equation}\label{eq.k}
    	v^{(k)}(x) \leq C(\|V\|_{L^1_\gamma(\rone)}), \ 0\leq k\leq \gamma-1,
    	\end{equation}
    	by induction in $k$. Since the inequalities above for $k=0$ is already established in \eqref{eq.1k0}, we suppose that the inequality \eqref{eq.k} holds for any $0\leq k \leq \gamma-1$ and so our goal will be to prove that
    	\begin{equation*}
    v^{(k+1)}(x)\leq C, \ k+1\leq \gamma-1.
    	\end{equation*}
    	
    	 The key tools here will be to consider the following formula
    	 \begin{equation}\label{eq.in1}
    	 \partial_\tau^{k+1} m_+(x,\tau) = \sum_{\ell=0}^{k+1} c_{k,\ell} \int_x^\infty \partial_\tau^{k+1-\ell} D(t-x,\tau) V(t) \partial_\tau^{\ell} m_+(t,\tau) dt
    	 \end{equation}
    	 and to quote the following inequalities
    	\begin{equation*}
    	\left| \partial_\tau^k D(t,x) \right| \leq C \min\left\{ \langle t\rangle^{k+1},\frac{\langle t\rangle^k}{|\tau|}\right\}, \ k=0,1,2,\dots,\ \tau \in \mathbb{C_+}.
    	\end{equation*}
    	Then, from the boundness of the following quantities\footnote{To prove that the quantity above are finite, we consider three different cases: $x<t<0$, $0<x<t$, $x<0<t$ separately. In the last case, we distinguish the behaviour if $x\approx t$, $|x|<<|t|$ and $|t|<<|x|$. }
    \begin{equation*}
    c_1=\sup_{t\geq x}\frac{\langle x_+\rangle^{\gamma-2-k}\langle t-x \rangle^{k+2}  }{\langle x_-\rangle^{2+k}  \langle t\rangle^\gamma }\in \rone_+,
    \end{equation*}
    \begin{equation*}
    c_2=\max_{0\leq \ell \leq k+1}\sup_{t\geq x}\frac{\langle x_+\rangle^{\gamma-2-k}\langle t-x \rangle^{k+2-\ell} \langle t_-\rangle^{1+\ell} }{\langle x_-\rangle^{2+k}  \langle t\rangle^\gamma\langle t_+\rangle^{\gamma-1-\ell} }\in \rone_+,
    \end{equation*}
combined with a Gronwall argument we get \eqref{eq.JImpr01}.

We do not prove the inequalities \eqref{eq.JImpr02}, \eqref{eq.JImpr03} and \eqref{eq.JImpr04} to avoid the repetition of the same arguments. We just note that for the proof of the inequalities \eqref{eq.JImpr03} and \eqref{eq.JImpr04} we use also the following estimate
	\begin{equation*}
	\|\tau^{1-k} D(t-x,\tau)\|_{C^{0,\sigma}(\mathbb{C}_+)}\leq C\langle t-x\rangle^{k+\sigma} ,\ \ k=0,1, \ \sigma \in [0,1).
	\end{equation*}
 \end{proof}

\subsection{Expansions for transmision and reflection coefficients}
The transmission coefficient
$T(\tau )$ and the reflection  coefficients
$R_\pm (\tau )$   are defined by  the formula
   \begin{equation}  \label{eq:kernel35} \begin{aligned} &    T(\tau )m_\mp (x ,\tau )= R_\pm (\tau )e^{\pm 2\im \tau x }m_\pm (x,\tau )+   m_\pm (x,-\tau ).
\end{aligned}
\end{equation}
 From \cite{DeiTru}  and from \cite{W} we have the following Lemma.
 \begin{lem}
\label{lem:TRcoeff} We have the following properties of the transmissions and reflection coefficients.
\begin{enumerate}[noitemsep,label=\alph*)]
   \item $T, R_\pm   \in C (\rone   )$.
   \item There exists $C_1,C_2>0$  such that:
    \begin{align} &  \label{eq:TRcoeff0}
		 |T(  \tau )-1| +| R_\pm (  \tau )  |\le C_1 \langle \tau  \rangle ^{ -1} \\&  
		\label{eq:TRcoeff}
     |T(\tau )|^2+ |R_{\pm }(\tau )|^2=1.
		\end{align}
\item If $T(0)=0,$ (i.e. zero is not a resonance point), then for some $\alpha \in \cone \setminus \{0\}$ and for some
$\alpha_+, \alpha_- \in \cone$
\begin{equation}\label{eq.TRa}
   T(\tau) = \alpha \tau + o(\tau), \  1+R_\pm(\tau) = \alpha_\pm \tau + o(\tau),
\end{equation}
for $\tau \in \rone,$ $ \tau \rightarrow 0.$
\end{enumerate}
\end{lem}
In particular,  \eqref{eq:TRcoeff},  \eqref{eq.TRa} follow from Sect.3  \cite{DeiTru}
and  \eqref{eq:TRcoeff0}
follows from  Theorem 2.3  \cite{W}.

The property  c) in the last Theorem suggest the following.
\begin{defn} \label{dres} The origin is a resonance point for the hamiltonian $\mathcal{H}$ if and only if
$$T(0) \neq 0.$$
\end{defn}

We can use the  assumption $V \in L^1_\gamma(\rone),$ $\gamma \geq 1,$ to get some more precise bounds.

 \begin{lem}
\label{lem:TRcoeffi} Suppose $V \in L^1_\gamma(\rone)$ with $\gamma \geq 1$  and $T(0)=0.$ Then  for any integer $k,$ $0 \leq k \leq \gamma -1$ we have:

\begin{enumerate}[noitemsep,label=\alph*)]
   \item $T, R_\pm   \in C^k (\rone   );$
   \item  There exists $C>0$  such that for any $\tau \in \rone $ we have:
  \begin{equation}\label{eq.TRC2}
 \left| \frac {d^k}{d\tau^k } T(\tau) \right| + \left|\frac {d^k} {d\tau^k } R_\pm(\tau)\right| \leq C, \
  \end{equation}
  \begin{equation}\label{eq.TRC21}
 \left| \frac {d^k}{d\tau^k } \left[\tau \left( T(\tau)-1\right) \right] \right| + \left|\frac {d^k} {d\tau^k } \left[ \tau R_\pm(\tau)\right]\right| \leq C. \
  \end{equation}
 \end{enumerate}
\end{lem}

\begin{proof}

The proof is based on the  relations
\begin{equation}\label{eq.63a}
    \frac{\tau}{T(\tau)} = \tau - \frac{1}{2i} \int_\rone V(t) m_+(t,\tau) dt, \ \  \tau \in  \R \setminus \{0\},
\end{equation}
  \begin{equation}\label{eq.TRC7}
    R_\pm (\tau) = \frac{T(\tau)}{2i\tau} \int_\rone e^{\mp 2it\tau} V(t) m_\mp(t,\tau) dt , \  \tau \in  \R \setminus \{0\}
  \end{equation}
    and the properties of the functions $m_\pm(t,\tau)$ from Lemma \ref{lem:JostIM}. Indeed we can write
    \begin{equation}\label{eq.Ttautau}
    \tau =T(\tau)\left(\tau - \frac{1}{2i} \int_\rone V(t) m_+(t,\tau) dt \right),
    \end{equation}
    and defer the boundness of the left sides in \eqref{eq.TRC2} and \eqref{eq.TRC21} from the relation above combined with the inequalities \eqref{eq.JImpr01} and \eqref{eq.JImpr02}.
\end{proof}
\begin{rem}
	It is easy to see that  \begin{equation}\label{eq.Btr}
	\left|\frac{T(\tau)}{\tau}\right|+\left|\frac{R_\pm(\tau)+1}{\tau}\right|\leq C
	\end{equation}
	and
	 \begin{equation}\label{eq.Btrsig}
	\left\|\frac{T(\tau)}{\tau}\right\|_{C^{0,\sigma}(\rone)}+\left\|\frac{R_\pm(\tau)+1}{\tau}\right\|_{C^{0,\sigma}(\rone)}\leq C.
	\end{equation}
	Indeed \eqref{eq.Btr} follows from relations \eqref{eq.Ttautau}, \eqref{eq.TRC7}, using the inequality \eqref{eq.1k0} and the property \eqref{eq.TRa} in Lemma \ref{lem:TRcoeff}. Similarly \eqref{eq.Btrsig} follows from relations \eqref{eq.Ttautau}, \eqref{eq.TRC7}, using the inequality \eqref{eq.3k0} and the property \eqref{eq.TRa} in Lemma \ref{lem:TRcoeff}.
\end{rem}

In the spirit of the Lemma before, we can establish the corresponding H\"older norm estimates for the transmission and the reflection coefficients.

\begin{lem}
\label{lem:TCH1} Suppose $V \in L^1_\gamma(\rone)$ with $\gamma > 1$  and $T(0)=0.$ Then  for any integer $k,$ $0 \leq k \leq \gamma -1$ and any
$\sigma \in (0,1)\cap(0,\gamma - 1-k]$ we have:
\begin{enumerate}[noitemsep,label=\alph*)]
   \item $T, R_\pm   \in C^{k,\sigma} (\rone   );$
   \item  There exists $C>0$  such that for any $\tau \in \rone $ we have:
  \begin{equation}\label{eq.TRC2h}
 \left\| \frac {d^k}{d\tau^k } T(\tau) \right\|_{C^{0,\sigma}(\rone)} + \left\|\frac {d^k} {d\tau^k } R_\pm(\tau)\right\|_{C^{0,\sigma}(\rone)} \leq C, \
  \end{equation}
  \begin{equation}\label{eq.TRC21h}
 \left\| \frac {d^k}{d\tau^k } \left[\tau \left( T(\tau)-1\right) \right] \right\|_{C^{0,\sigma}(\rone)} + \left\|\frac {d^k} {d\tau^k } \left[ \tau R_\pm(\tau)\right]\right\|_{C^{0,\sigma}(\rone)} \leq C. \
  \end{equation}
 \end{enumerate}
\end{lem}

\section{Appendix I: Gronwall's lemma on the real line.}

In this section we shall recall first some of modifications of the classical Gronwall's inequality on  $\rone.$

\begin{lem} \label{l.a2.1rone}
If $v(x), a(x), b(x) $ are continuous non negative functions on $\rone,$ and for any real $r$ we have
\begin{equation}\label{eq.GR1}
   a(x), v(x) \in L^\infty((r,\infty)), b(x) \in L^1((r,\infty))
\end{equation}
  that satisfy the inequality
\begin{equation}\label{eq.a2.m1rone}
    v(x) \leq a(x)  + \int_x^\infty b(t) v(t) dt
\end{equation}
then we have
\begin{equation}\label{eq.a2.m3rone}
  v(x) \leq a(x) + \int_x^\infty a(t) b(t)  \exp\left(\int_x^t b(s) ds \right) dt.
\end{equation}
\end{lem}
\begin{proof} We shall sketch the proof for completeness.
Set
$$\varphi(x) = \int_x^\infty b(t) v(t) dt .$$
The function is well-defined and $C^1$ due to the assumption \eqref{eq.GR1}.
Then $$ \varphi^\prime(x) = - b(x) v(x) \geq -b(x)(\varphi(x) + a(x) )$$
and
$$ \left( e^{-B(x)} \varphi(x) \right)^\prime \geq - e^{-B(x)}b(x) a(x)$$
with $B(x) = \int_x^\infty b(t)dt.$
Integrating this inequality in the interval $(x, R)$, we get
$$ \varphi(x) \leq e^{B(x)-B(R)} \varphi(R) + \int_x^R e^{B(x)-B(t)} a(t)b(t) dt.$$
Using again the assumption \eqref{eq.GR1}, we see that $$ \lim_{R \nearrow \infty} B(R) = 0, \ \lim_{R \nearrow \infty} \varphi(R) = 0 $$
so we get
$$ \varphi(x) \leq \int_x^\infty e^{B(x)-B(t)} a(t)b(t) dt.$$
Then \eqref{eq.a2.m1rone} implies $v(x) \leq a(x) + \varphi(x)$ and we arrive at \eqref{eq.a2.m3rone}.
This completes the proof.

\end{proof}

\begin{cor} \label{c.a2.1rone}
If $a(x)$ is a continuous $L^\infty(\R)$ function, such that
$$ a(x) = \left\{
            \begin{array}{ll}
              C>0, & \hbox{for $x \leq 0$;} \\
             \mbox{decreasing positive function }, & \hbox{for $x >0,$}
            \end{array}
          \right.
$$
and $b \in L^1(\R),$ then the inequality \eqref{eq.a2.m1rone} implies
$$ v(x) \leq C a(x_+), \ \ x_+ = \max(0,x).$$
\end{cor}

\bibliographystyle{amsplain}

\end{document}